\documentclass[11pt,a4paper,english,reqno]{amsart}
\usepackage{amsmath,amssymb,amsfonts,epsfig,mathrsfs}
\usepackage[T1]{fontenc}

\usepackage{color}
\usepackage{array}
\usepackage{amsthm}
\usepackage{amstext}
\usepackage{graphicx}
\usepackage{setspace}
\usepackage[margin=2.5cm]{geometry}
\usepackage{bbm}
\usepackage{color}
\usepackage{enumitem}
\usepackage{undertilde}
\allowdisplaybreaks[4]

\usepackage{amscd,psfrag}
\usepackage{yhmath}
\usepackage[mathscr]{eucal}

\usepackage{slashed}

\makeatletter
\pdfpageheight\paperheight
\pdfpagewidth\paperwidth

\allowdisplaybreaks[4]

\usepackage{epstopdf}
\usepackage{chngcntr}
\counterwithin{figure}{section}
\usepackage{mathrsfs}

\usepackage{indentfirst}	

\usepackage[normalem]{ulem}
\theoremstyle{plain}

\newtheorem{definition}{Definition}[section]
\newtheorem{theorem}[definition]{Theorem}
\newtheorem*{theorem*}{Theorem}

\newtheorem{assumption}[definition]{Assumption}
\newtheorem{remark}[definition]{Remark}

\newtheorem*{remark*}{Remark}
\newtheorem*{sideremark*}{Side Remark}\newtheorem*{mt*}{Main Theorem}

\newtheorem*{claim*}{Claim}
\newtheorem*{q*}{Question}
\newtheorem{lemma}[definition]{Lemma}
\newtheorem{corollary}[definition]{Corollary}
\newtheorem*{corollary*}{Corollary}
\newtheorem*{proposition*}{Proposition}

\newtheorem{proposition}[definition]{Proposition}

\newcommand{\R}{\mathbb{R}}
\newcommand{\C}{\mathbb{C}}
\newcommand{\na}{\nabla}
\newcommand{\lie}{\mathcal{L}}
\newcommand{\dd}{{\rm d}}
\newcommand{\p}{\partial}
\newcommand{\e}{\epsilon}
\newcommand{\emb}{\hookrightarrow}

\newcommand{\map}{\rightarrow}
\newcommand{\G}{\Gamma}

\newcommand{\1}{\mathbbm{1}}
\newcommand{\xyhat}{\widehat{x-y}}

\newcommand{\two}{{\rm II}}

\newcommand{\scal}{\mathcal{S}}
\newcommand{\what}{\widehat{\omega}}

\newcommand{\htwo}{\mathcal{H}^2}
\newcommand{\GG}{{\mathcal{G}}}
\newcommand{\good}{{\rm good}}
\newcommand{\bad}{{\rm bad}}
\newcommand{\NN}{\mathcal{N}}

\newcommand{\bbi}{\mathbf{b}^{(i)}}
\newcommand{\n}{\mathbf{n}}
\newcommand{\yst}{{y^{\star}}}
\newcommand{\tyst}{(Ty)^\star}
\newcommand{\TI}{\Theta^{(i)}}
\newcommand{\I}{\mathcal{I}}
\newcommand{\B}{\mathcal{B}}
\newcommand{\OO}{\mathcal{O}}
\newcommand{\smallo}{\mathfrak{o}}
\newcommand{\F}{\mathcal{F}}

\newcommand{\txty}{{Tx-Ty}}
\newcommand{\zzz}{{\zeta\bigg(\frac{|\txty|}{d_4}\bigg)}}
\newcommand{\sh}{{\Psi^{\sharp}}}
\newcommand{\fl}{{\Psi^{\flat}}}

\newcommand{\shh}{\widetilde{\Psi}^\sharp}
\newcommand{\fll}{\widetilde{\Psi}^\flat}
\newcommand{\txtys}{Tx-\tyst}
\newcommand{\shhh}{\overline{\Psi}^\sharp}

\newcommand{\ooo}{{O_{d_3}(U_b)}}

\allowdisplaybreaks[4]

\def\XXint#1#2#3{{\setbox0=\hbox{$#1{#2#3}{\int}$ }
\vcenter{\hbox{$#2#3$ }}\kern-.6\wd0}}

\numberwithin{equation}{section}
\numberwithin{figure}{section}

\keywords{Navier--Stokes equations; incompressible; vorticity; regularity; weak solution; oblique derivatives; boundary effects; Green's matrix; vortex stretching; vortex alignment; Navier boundary condition.}

\subjclass[2010]{35Q30, 76D05, 76N10, 35J08}

\title{Geometric Regularity Criteria for incompressible Navier--Stokes Equations with Navier Boundary Conditions}

\author{Siran Li}
\address{Siran Li: Department of Mathematics, Rice University, MS 136
P.O. Box 1892, Houston, Texas, 77251-1892, USA; \, $\bullet$ \,  Department of Mathematics, McGill University, Burnside Hall, 805 Sherbrooke Street West, Montreal, Quebec, H3A 0B9, Canada.}

\email{\texttt{Siran.Li@rice.edu}}

\date{\today}

\begin{document}

\maketitle

\begin{abstract}
We study the regularity criteria for weak solutions to the 3D incompressible Navier--Stokes equations in terms of the direction of vorticity, taking into account the boundary conditions. A boundary regularity theorem is proved on regular curvilinear domains with a family of oblique derivative boundary conditions, provided that the directions of vorticity are coherently aligned up to the boundary. As an application, we establish the boundary regularity for weak solutions to Navier--Stokes equations in round balls, half-spaces and right circular cylindrical ducts, subject to the classical Navier and kinematic boundary conditions. 
\end{abstract}

\section{Introduction and Statement of Main Results}

This paper is concerned with the regularity of weak solutions to the 3-dimensional incompressible Navier--Stokes equations on a regular domain $\Omega$ in $\R^3$:
\begin{eqnarray}
&&\p_t u + {\rm div}\, (u\otimes u) - \nu \Delta u + \na p =0 \qquad \text{ in } [0,T^\star[ \times \Omega,\label{ns}\\
&&{\rm div}\, u = 0 \qquad \text{ in } [0,T^\star[ \times \Omega,\label{incompressible}\\
&& u|_{t=0} = u_0 \qquad\text{ on } \{0\}\times \Omega.\label{initial data}
\end{eqnarray}
The fluid boundary $\p\Omega =: \Sigma$ is a regular surface (at least $C^2$). Here $u:\Omega \map \R^3$ is the velocity, $p:\Omega \map \R$ the pressure, and $\nu>0$ the viscosity of the fluid. We  study the regularity criteria {\em up to the boundary} under the assumptions on the geometry of vorticity alignment. The system \eqref{ns}\eqref{incompressible}\eqref{initial data} will be considered under a general class of boundary conditions.

Let us begin the discussion on boundary conditions with some motivating examples: Take $\Omega$ to be a round ball, a half-space or a  cylindrical duct smoothly embedded in $\R^3$. Then, we impose to Eqs.\,\eqref{ns}\eqref{incompressible}\eqref{initial data} the classical Navier and kinematic boundary conditions: Let $\mathbb{T} \in \mathfrak{gl}(3;\R)$ be the {\em Cauchy stress tensor} of the fluid in $\Omega$ (here and throughout $\mathfrak{gl}(3,\R)$ denotes the space of $3 \times 3$ real matrices), defined by 
\begin{equation}\label{cauchy stress}
\mathbb{T}^i_j := \nu (\na_i u^j + \na_j u^i) \qquad \text{ for } i,j\in\{1,2,3\}.
\end{equation}
Its contraction with the normal vector field on $\Sigma$, known as the {\em Cauchy stress vector} $\mathbf{t}\in \G(T\Sigma)$, describes the stress on the boundary contributed by the fluid from the normal direction:
\begin{equation}
\mathbf{t}^i := \sum_{j=1}^3\mathbb{T}^i_j \n^j \qquad \text{ for } i\in\{1,2,3\}.
\end{equation}
The classical Navier boundary condition, first proposed by Navier \cite{navier} in 1816, requires the tangential component of the Cauchy stress vector to be  proportional to the tangential component of the velocity on $\Sigma$:
\begin{equation}\label{navier bc}
\beta u \cdot \tau + \mathbf{t} \cdot \tau = 0 \qquad \text{ for each } \tau \in \G(T\Sigma) \text{ on } [0,T^\star[ \times \Sigma, 
\end{equation}
where the constant $\beta > 0$ is known as the {\em slip length} of the fluid. Here and in the sequel we write $\G(T\Sigma)$ for the space of tangent vector fields to $\Sigma$.   We moreover impose the {\em kinematic} or {\em impenetrability} boundary condition:
\begin{equation}\label{kinematic bc}
u \cdot \n = 0 \qquad \text{ on } [0,T^\star[ \times \Sigma,
\end{equation}
where $\n$ is the outward unit normal vector field along $\Sigma$. The above choices for domains $\Omega$ and boundary conditions all have physical relevance.

Throughout the paper, we say that $u$ is a {\em weak solution} to the Navier--Stokes equations \eqref{ns}\eqref{incompressible} if 
$$
u\in L^\infty(0,T^\star; L^2(\Omega; \R^3))\cap L^2(0,T^\star; H^1(\Omega; \R^3))
$$ satisfies the equations in the sense of distributions, and, in addition, the {\em energy inequality} holds:
\begin{equation}\label{energy ineq}
\frac{1}{2}\frac{\dd }{\dd t}\int_\Omega |u(t,x)|^2\,\dd x + \nu \int_\Omega |\na u(t,x)|^2\,\dd x - c\int_\Sigma |u(t,y)|^2\,\dd \htwo(y) \leq 0\quad \text{ for each } t\in [0,T^\star[,
\end{equation}
where $c$ is a constant depending only on $\Omega$ and $\nu$. The initial condition \eqref{initial data} is also understood in the sense of distributions.  The energy inequality was proposed in the classical works by Leray \cite{leray} and Hopf \cite{hopf} on Eqs.\,\eqref{ns}\eqref{incompressible} in $\Omega = \R^3$, where $c =0$. Here the $c$ term is introduced to account for the boundary conditions; we shall give a justification in Lemma \ref{lemma: energy ineq} in Sect.\,3 below. A weak solution $u$ is said to be a {\em strong solution} if it further satisfies 
$$
\na u \in L^\infty(0, T^\star; L^2(\Omega; \mathfrak{gl}(3,\R))) \cap L^2(0, T^\star; H^1(\Omega; \mathfrak{gl}(3,\R))).
$$
We adopt the above definitions for weak and strong solutions also for more general types of boundary conditions, {\it e.g.}, the oblique derivative boundary condition in \eqref{homog bc}, as well as the Navier and kinematic boundary conditions in \eqref{navier bc}\eqref{kinematic bc}.

In regard to the aforementioned motivating examples, our main result of the paper can be stated as follows:
\begin{theorem}\label{thm: main}
Let $\Omega \subset \R^3$ be one of the following smooth domains: a round ball, a half-space, or a right  circular cylindrical duct.  Let $u$ be a weak solution to the Navier--Stokes equations \eqref{ns}\eqref{incompressible}\eqref{initial data} with the Navier and kinematic boundary conditions \eqref{navier bc}\eqref{kinematic bc}. Suppose that the vorticity $\omega = \na \times u$ is {\em coherently aligned up to the boundary} in the following sense: there exists a constant $\rho>0$ such that
\begin{equation}
|\sin \theta(t; x,y)| \leq \rho \sqrt{|x-y|} \qquad \text{ for all $x,y\in\overline{\Omega}, \,  t<T^\star$}.
\end{equation}
Here, the turning angle of vorticity $\theta$ is defined as
\begin{equation}
\theta(t;x,y):=\angle\big(\omega(t,x),\omega(t,y)\big).
\end{equation}
Then $u$ is a strong solution on $[0,T^\star[$.
\end{theorem}

\begin{remark}
For $y \in \Sigma$, $\omega(t,y)$ is understood in the sense of trace. 
\end{remark}

The regularity theory for the incompressible Navier--Stokes equations has long been a central topic in PDE and mathematical hydrodynamics; {\it cf.} Constantin--Foias \cite{cfs}, Fefferman \cite{f}, Lemari\'{e}-Rieusset \cite{lr}, Temam \cite{temam}, Seregin \cite{s} and many references cited therein. One major problem in the regularity theory is concerned with the regularity of weak solutions, {\it i.e.}, under what conditions can a weak solution be the strong solution. In \cite{cf} Constantin and Fefferman first proposed the following {\em geometric regularity condition}: For a weak solution to the Navier--Stokes equations on the {\em whole space} $\Omega = \R^3$, if there are constants $\rho, \Lambda>0$ such that 
\begin{equation}\label{cf1}
|\sin \theta(t; x,y)|\1_{\{|\omega(t,x)|\geq \Lambda, |\omega(t,y)| \geq \Lambda\}} \leq \rho{|x-y|} \qquad \text{ for all $x,y\in\R^3$, $t<T^\star$},
\end{equation}
then the weak solution is indeed  strong. Here $\theta$ is the turning angle of vorticity as in Theorem \ref{thm: main}. The above result by  Constantin--Fefferman \cite{cf} suggests that, if the vortex lines of the fluid are coherently aligned, {\it i.e.}, without sharp turnings before time $T^\star$, then the weak solutions cannot blow up by $T^\star$. It opens up the ways for many subsequent works on regularity conditions in terms of the geometry of vortex structures; see Beir\~{a}o da Veiga--Berselli \cite{bb1, bb2}, Beir\~{a}o da Veiga \cite{bb3}, Chae \cite{chae}, Li \cite{l}, Giga--Miura \cite{gm}, Gruji\'{c} \cite{grujic}, Gruji\'{c}--Ruzmaikina \cite{gr}, Vasseur \cite{v} and many others. Let us remark that, in \cite{bb1}, Beir\~{a}o da Veiga--Berselli improved the right-hand side of \eqref{cf1} in Constantin--Fefferman's criterion to $\rho\sqrt{|x-y|}$; that is, they improved the H\"{o}lder exponent from $1$ to $1/2$.

In line with the above results, Theorem \ref{thm: main} proposes a geometric regularity condition for the weak solutions to the Navier--Stokes equations. The main new feature is that we work on regular domains $\Omega \subset \R^3$, so the boundary conditions play a crucial role when we investigate the regularity theory {\em up to the boundary}. In the literature, the ``geometric boundary regularity conditions'' have been studied for only one special slip-type boundary condition proposed by Solonnikov--\u{S}\u{c}adilov in \cite{ss} (also see Xiao--Xin \cite{xx}), which agrees with the Navier and kinematic boundary conditions \eqref{navier bc}\eqref{kinematic bc} if and only if $\Omega = \R^3_+$:
\begin{equation}\label{special bc}
u \cdot \n = 0, \qquad \omega \times \n = 0 \qquad \text{ on } [0,T^\star[ \times \Sigma;
\end{equation}
see Beir\~{a}o da Veiga \cite{bb3} for the case of $\Omega = \R^3_+$ and Beir\~{a}o da Veiga--Berselli \cite{bb2} for the case of general bounded $C^{3,\alpha}$ domains $\Omega \Subset \R^3$. Let us note that,  in the latter case, the condition \eqref{special bc} no longer agrees with the Navier and kinematic boundary conditions. Therefore, our work is the first in the literature to prove the geometric boundary regularity under the physical (Navier and kinematic) boundary conditions on regular curvilinear domains.

Let us briefly remark on the Navier and kinematic boundary conditions in Eq.\,\eqref{navier bc}\eqref{kinematic bc}. The kinematic boundary condition requires that the fluid motion on $\Sigma$ can only be tangential with respect to the boundary, {\it i.e.,} $\Sigma$ is impermeable. The Navier boundary condition further describes the tangential motion of the fluid on $\Sigma$: its velocity is proportional to the tangential component of the Cauchy stress vector $\mathbf{t}$. It was proposed by Navier \cite{navier} to resolve the incompatibility between the theoretical predictions from the Dirichlet boundary condition ($u = 0$ on $\Sigma$) and the experimental data. It was later considered by Maxwell in 1879 (\cite{maxwell}) for the motion of rarefied gases. In recent years, the Navier boundary condition has been extensively studied in fluid models when the {\em curvature effect} of the boundary becomes considerable. In particular, free capillary boundaries, perforated boundaries or the presence of an exterior electric field may lead to such situations for flows with large Reynolds number; {\it cf.} Achdou--Pironneau--Valentin \cite{apv}, B\"{a}nsch \cite{bansch}, Einzel--Panzer--Liu \cite{epl} and many others for related physical and numerical studies, and {\it cf.} Berselli--Spirito \cite{bs}, Chen--Qian \cite{cq}, Iftimie--Raugel--Sell \cite{irs}, J\"{a}ger--Mikeli\'{c} \cite{jm}, Masmoudi--Rousset \cite{mr}, Neustupa--Penel \cite{np}, Xiao--Xin \cite{xx} and many references cited therein for the mathematical analysis of the Navier boundary condition.

Our strategy for proving Theorem \ref{thm: main} is as follows. By elementary energy estimates (see Sect.\,3) it suffices to control the {\em vortex stretching term}:
\begin{equation}\label{stretch}
[{\rm Stretch}] := \Big|\int_\Omega \scal u(t,x):\omega(t,x)\otimes \omega(t,x)\,\dd x \Big|,
\end{equation}
where $\scal u$ is the rate-of-strain tensor, {\it i.e.}, the symmetrised gradient of $u$:
\begin{equation}
\scal u:= \frac{\na u + \na^\top u}{2} : [0,T^\star[ \times \Omega \map \mathfrak{gl}(3,\R). 
\end{equation}
For this purpose, following \cite{cf} we represent $\scal u$ by a singular integral of $\omega$. We first localise the problem to coordinate charts on $\overline{\Omega}$ ({\it cf.} Gruji\'{c} \cite{grujic}). In the interior charts the integral kernel ``looks like'' that on $\R^3$, whose estimates are obtained by Constantin--Fefferman in \cite{cf}. In each boundary chart, thanks to the results by Solonnikov \cite{s1, s2}, there exists one single Green's matrix for the Laplacian, which can be explicitly constructed by transforming to the model problem (Poisson equation with oblique derivative boundary conditions; {\it cf.} Sect.\,2 below) on the half space $\R^3_+$. With suitable bounds for the term $[Stretch]$ at hand (these estimates occupy the major part of the paper; see Sect.\,4 below), we can conclude using the Hardy--Littlewood--Sobolev inequality and the Gr\"{o}nwall's lemma.

In the estimation of $[Stretch]$, one major difficulty is to control the boundary terms, which naturally arise during the integration by parts. We realise that if the vorticity turning angle $\theta$ remains coherently aligned up to the boundary (as in the assumption in Theorem \ref{thm: main}), then, thanks to the geometric structure of the boundary terms, such bounds can be achieved. Our assumption is weaker than that by Beir\~{a}o da Veiga--Berselli in \cite{bb2}: it is required in \cite{bb2} that $\omega \times \n = 0$, {\it i.e.}, $\omega$ points in the normal direction to the regular hypersurface $\Sigma \subset \R^3$, which is automatically coherently aligned on the boundary $\Sigma$.  (Indeed, when $\omega \times \n = 0$ the boundary term in $[Stretch]$ vanishes.) On the other hand, in each boundary chart we need to straighten the boundary by a local $C^2$-diffeomorphism onto some subset of $\R^3_+$. These boundary-straightening diffeomorphisms enter the estimates in a crucial way. We need delicate analyses for the geometry of $\Sigma$ to bound the contributions to $[Stretch]$ from the boundary charts. Many of these estimates are new to the literature.

Moreover, let us emphasise that our approach in this paper applies to more general boundary conditions than those considered in Theorem \ref{thm: main}:
\begin{enumerate}
\item
The energy estimates in Sect.\,3 below are valid for Navier and kinematic boundary on arbitrary regular embedded surfaces in $\R^3$;
\item
The potential estimates is applicable to the diagonal oblique derivative boundary conditions with constant coefficients (see Sect.\,4).
\end{enumerate}
In both (1) and (2) above, we do not need to impose any restriction on the specific geometry of $\Omega$ other than sufficient regularity requirements, {\it e.g.}, $\Omega \in C^{3,\alpha}$.



\medskip

The remaining parts of the paper is organised as follows: In Sect.\,2 we present Solonnikov's theory on the  Green's matrices for a special class of elliptic systems. Next, in Sect.\,3 we collect the energy estimates for the Navier--Stokes system \,\eqref{ns}\eqref{incompressible}\eqref{navier bc}\eqref{kinematic bc}. In Sect.\,4 we prove the boundary regularity theorem for the Navier--Stokes equations under the general diagonal oblique derivative conditions. This is achieved by potential estimates based on the theory outlined in Sect.\,2. Finally, in Sect.\,5, we deduce Theorem \ref{thm: main} for the Navier and kinematic boundary conditions as an instance of the theory laid down in Sect.\,4.

\section{Green's Matrices}\label{Green's Matrices}

In this section we summarise the theory of Green's matrices for a general family of boundary value problems for the diagonal elliptic systems. It is the foundation of the subsequent developments in the paper. For the convenience of exposition we focus only on the $(3 \times 3)$ elliptic systems, although the general theory applies to $N \times M$ systems for arbitrary $N,M \geq 2$. 

Let us consider the system with the homogeneous boundary conditions:
\begin{eqnarray}
&& -\Delta u = f:= \na \times \omega \qquad \text{ in } [0,T^*[ \times \Omega,\label{Poisson eq} \\
&& (\NN u)^i= a^{(i)} u^i + \sum_{j=1}^3 b^{(i)}_{j}\na_j u^i = 0 \qquad \text{ on } [0,T^*[ \times \p\Omega\quad\text{ for each } i=1,2,3,\label{homog bc}
\end{eqnarray}
where without loss of generality we assume 
\begin{equation}
a^{(i)} \leq 0,\qquad \sum_{j=1}^3 \big[b^{(i)}_j\big]^2=1
\end{equation} 
for each $i=1,2,3$ and $\NN u =\{(\NN u)^i\}_{i=1}^3$, in some local coordinates $\{x^1, x^2, x^3\}$ near a point $p\in\Sigma := \p\Omega$. The key assumption here is that the boundary conditions \eqref{homog bc} are {\em diagonal}: in suitable coordinates it is decoupled into three scalar equations in $u^1, u^2$ and $u^3$, respectively. This ensures that the Green's matrices for Problem \eqref{Poisson eq}\eqref{homog bc}, constructed by Solonnikov (\cite{s1,s2}; see below for details), are diagonal. Also, in order to write down the explicit expressions for the Green's matrices, we require that
$$
a^{(i)}, \, \bbi \text{ are constants  for each }  i\in\{1,2,3\}.
$$
Our goal is to represent $u$ in terms of $\omega$; in the case of $\Omega= \R^3$ and no boundary conditions other than suitable decay at infinity, the above system is solved by the convolution $u=K_{\rm bs}\ast \omega$, where $K_{\rm bs}$ is the classical Biot--Savart kernel.

The system \eqref{Poisson eq}\eqref{homog bc} is known as an {\em oblique derivative problem} for the Poisson equation. Throughout we write $\Sigma:=\p\Omega$ and $\n:=$ the outward unit normal vector field along $\Sigma$. Introducing the notations
$$
\bbi := \big(b^{(i)}_1,b^{(i)}_2,b^{(i)}_3\big)^\top\qquad \text{ for } i=1,2,3
$$
and writing 
$$
\na = (\p/\p x^1, \p/\p x^2, \p/\p x^3),
$$ 
we can rewrite the boundary condition \eqref{homog bc} as 
\begin{equation*}
(\NN u)^i = a^{(i)}u^i + \bbi \cdot \na u^i = 0 \qquad \text{ on } \Sigma \text{ for each } i=1,2,3.
\end{equation*}
We note that the boundary condition \eqref{homog bc} is fairly general: when $\bbi =0$, it reduces to the Dirichlet boundary condition for $u^i$; when $\bbi = (\bbi \cdot \n) \n$ and $a^{(i)}=0$, it reduces to the Neumann boundary condition; moreover, when $\bbi\cdot\n \neq 0$, the condition \eqref{homog bc} is known as the regular oblique derivative condition.

We shall divide our discussions on the boundary value problem \eqref{Poisson eq}\eqref{homog bc} in two subsections: In Sect.\,2.1 we collect some facts about elliptic PDE systems, and in Sect.\,2.2 we present the Green's matrices associated to the oblique derivative boundary conditions.

\subsection{Elliptic Systems and the Existence of Green Matrices}

In this subsection we outline the theory for the elliptic systems of the Petrovsky type developed by Solonnikov \cite{s1, s2}. Our use of Solonnikov's theory is motivated by \cite{bb2} by Beir\~{a}o da Veiga--Berselli; also see Proposition 2.2 in Temam \cite{temam}.

We consider a $3\times 3$ linear PDE system
\begin{equation}\label{3 by 3 system}
\lie_i u := \sum_{j=1}^3 l_{ij}(x,\na) u^j = f_i \qquad \text{ in } \Omega \subset \R^3
\end{equation}
which is {\em elliptic} in the sense of ADN theory (Agmon--Douglis--Nirenberg \cite{adn1, adn2}). Here $i,j\in\{1,2,3\}$, $u=(u^1,u^2,u^3), f=(f_1,f_2,f_3):\Omega\map\R^3$ are vector fields/one-forms on $\Omega$, and $\{l_{ij}\}$ is a $3 \times 3$ matrix of differential operators. A family of weights $\{s_1,s_2,s_3; t_1,t_2,t_3\}\subset\mathbb{Z}$ is associated to the system \eqref{3 by 3 system}, such that
\begin{equation}\label{weights}
s_i \leq 0 \text{ for each } i, \qquad \text{ the order of $l_{ij}$} \leq  \max\{0, s_i+t_j\}.
\end{equation}
Then, we set $l'_{ij}(x,\na)$ to be the principal part of $l_{ij}$, namely the sum of all terms in $l_{ij}(x,\na)$ of order $(s_i+t_j)$, and consider the characteristic matrix $\{l'_{ij}(x,\xi)\}_{1\leq i,j\leq 3}$. Then, \eqref{3 by 3 system} is {\em elliptic} if and only if $s_i, t_j$ satisfying \eqref{weights} exist for every $x\in\Omega$, and that
\begin{equation}\label{det}
\det \big\{l'_{ij}(x,\xi)\big\} \neq 0 \qquad \text{ for all } \xi \in \R^3 \setminus \{0\}.
\end{equation}

Now we consider the boundary conditions imposed to the system \,\eqref{3 by 3 system}. Throughout, $\Sigma:=\p\Omega$ is a $C^2$ surface, and we use $p$ to denote a typical boundary point on $\Sigma$. A generic (linear) boundary condition is of the form
\begin{equation}\label{generic bc}
\sum_{j=1}^3 B_{hj}(p,\na) u_j(p) = \phi_h(p) \qquad \text{ on } \Sigma \text{ for } h=1,2,\ldots, m,
\end{equation}
where 
\begin{equation}
m:=\frac{1}{2}\deg\,\det\big\{l'_{ij}(p,\xi)\big\} > 0,
\end{equation}
for which the determinant (as in Eq.\,\eqref{det}) is viewed as a polynomial in $\xi$. Similarly, viewing $B_{hj}(p,\xi)$ as a $\C$-coefficient polynomial in $\xi$ (depending on $p$), we consider another set of weights $\{r_1,r_2, \ldots, r_m\}\subset \mathbb{Z}$ such that
\begin{equation}\label{weights 2}
\deg\big\{B_{hj}(p,\xi)\big\} \leq \max\{r_h + t_j, 0\}
\end{equation}
with $t_j$ given as above. Now, for any $p \in \Sigma$ we consider $\Xi \in T_p\Sigma \setminus\{0\}$ and 
\begin{eqnarray}
&&
\tau_k^+ (p,\Xi) := \text{ roots in $\tau$ with positive imaginary part of $\lie_k(p,\Xi+ \tau \n)=0$},\\
&&
M^+ (p,\Xi,\tau) := \prod_{h=1}^m \Big(\tau-\tau^+_h(p,\Xi)\Big).
\end{eqnarray}
We also write $\{B'_{hj}\}$ for the principal part of $B_{hj}$, and view $M^+(p,\Xi,\tau)$ as a polynomial in $\tau$. The boundary condition \eqref{generic bc} is said to be {\em complementing} to the elliptic system \eqref{3 by 3 system} if, for every $p\in\Sigma$ and every $\Xi \in T_p\Sigma \setminus\{0\}$, there exist $\{r_h\}_{h=1,2,\ldots, m}$ everywhere satisfying \eqref{weights 2}, and
\begin{equation}
\sum_{h=1}^m C_h \sum_{j=1}^3 B'_{hj} \Big\{\text{adjoint matrix of } l_{ij}'(p,\Xi + \tau\n)\Big\} \equiv 0 \, ({\rm mod}\, M^+ ) \Leftrightarrow C_h=0 \text{ for all h}.
\end{equation}

All the classical boundary conditions (Dirichlet, Neumann, regular oblique derivative etc., homogeneous or inhomogeneous) are known to be complementing to the Poisson equation. 
\begin{definition}
Consider the elliptic PDE system \eqref{3 by 3 system}\eqref{generic bc} with complementing boundary conditions in the ADN sense, and with weights $\{s_i, t_j, r_h\}$ as above. If one can choose $s_i=0$ and $r_h<0$ for all $i\in\{1,2,3\}$ and $h\in\{1,\ldots,m\}$, then \eqref{3 by 3 system}\eqref{generic bc} is said to be of the Petrovsky type.
\end{definition}

\begin{lemma}
The system \eqref{Poisson eq}\eqref{homog bc} is of the Petrovsky type.
\end{lemma}

\begin{proof}
In this case we have $\lie=-\Delta
$ and $l'_{ij}(x,\xi)=(\xi^1)^2(\xi^2)^2(\xi^3)^2$, hence $m=3$. Using $\NN = \{B_{hj}\}_{1\leq h,j\leq 3}$ in \eqref{homog bc}, we can pick $s_1=s_2=s_3=0$, $t_1=t_2=t_3=2$ and $r_1=r_2=r_3=-1$.   \end{proof}
	
	Therefore, in view of Solonnikov's theory on the existence of Green's matrices for Petrovsky-type elliptic systems ({\it cf.} p126, \cite{s2} and p606, \cite{bb2}), we may deduce:
\begin{lemma}\label{lemma: solonnikov}
A matrix field $\{\GG_{ij}\}_{1\leq i,j\leq 3}: \Omega \times \Omega \map \R$ exists for the system \eqref{Poisson eq}\eqref{homog bc} such that 
\begin{equation}
u^i(x)=\sum_{j=1}^3\int_\Omega \GG_{ij}(x,y)f^j(y)\,\dd y \qquad \text{ for each } i=1,2,3.
\end{equation}
Moreover, we have the decomposition $\GG = \GG^\good + \GG^\bad$, where
\begin{equation}
\exists \,C_\bad >0: \qquad \Big|\na^\alpha_x\na^\beta_y\GG^\bad(x,y)\Big| \leq \frac{C_\bad}{|x-y|^{|\alpha|+|\beta|+1}}  \quad\text{ for all } x\neq y\in\Omega,
\end{equation}
and
\begin{equation}\label{good}
\exists\, C_\good >0, \delta>0:\qquad \Big|\na^\alpha_x\na^\beta_y\GG^\good(x,y)\Big|\leq\frac{C_\good}{|x-y|^{|\alpha|+|\beta|+1-\delta}}\quad\text{ for all } x\neq y\in\Omega,
\end{equation}
for any multi-indices $\alpha,\beta \in \mathbb{N}^\mathbb{N}$. For $\Sigma=\p\Omega$ sufficiently regular, one can take $\delta>1/2$. 
\end{lemma}
	In the lemma above, $\GG$ is known as the {\em Green's matrix} for the oblique derivative boundary value problem for the Poisson equation \eqref{Poisson eq}\eqref{homog bc}. The crucial point is that the solution can be represented by {\em one single} matrix. Moreover, under our assumption that the boundary conditions are diagonal (decoupled), we know that $\GG$ is a diagonal matrix, namely
	\begin{equation}
	\GG_{ij}(x,y) = g(x,y) \delta_{ij}
	\end{equation}
for a scalar function  $g:\Omega \times \Omega \map \R$. In this case 
$$
u^i(x)=\int_\Omega g(x,y)f^i(y)\,\dd y,
$$
so we can carry out potential estimates for the corresponding {\em scalar} functions. Thus we can resort to well-developed theories in PDE; {\it cf.} Gilbarg--Trudinger \cite{gt}.

\subsection{Diagonal Oblique Derivative Boundary Conditions}
Now, let us discuss the system \eqref{Poisson eq}\eqref{homog bc} on the half space $\R^3_+:=\{(x^1,x^2,x^3)\in\R^3: x^3>0\}$, namely
\begin{eqnarray}
&&-\Delta u = f \qquad \text{ in } \R^3_+,\label{model system}\\
&&\NN u^i = a^{(i)} u^i + \sum_{j=1}^3 b^{(i)}_j \na_j u^i = 0 \qquad \text{ on } \{x^3=0\} \text{ for each } i =1,2,3,\label{model system bc}
\end{eqnarray}
where $a^{(i)}, \bbi$ are {\em constants}. In Sect.\,4 below we shall localise \eqref{Poisson eq}\eqref{homog bc} so that, in each chart near the boundary $\Sigma$,  the system ``looks like'' the above model system \eqref{model system}\eqref{model system bc} ({\it i.e.}, modulo certain linear transforms which can be nicely controlled). 

For  each $y \in \R^3_+$ let us write:
\begin{equation*}
y=(y',y^3) \text{ where } y'=(y^1, y^2), \qquad \yst := (y', -y^3).
\end{equation*}
That is, $\yst$ is the reflected point (the ``virtual charge'') across the boundary $\{x^3=0\}$. We use $\langle\cdot,\cdot\rangle$ to denote the Euclidean inner product. Also, for $x,y \in \R^3$ we write
\begin{equation}
\G(x,y):=\frac{1}{|x-y|},
\end{equation}
namely the fundamental solution to the Laplace equation in $\R^3$ (up to a multiplicative constant). One also denotes by
\begin{equation}
\xi := \frac{x-\yst}{|x-\yst|} \qquad \text{ for } x, y\in\R^3_+.
\end{equation}
Then, following Sect.\,6.7 in Gilbarg--Trudinger \cite{gt}, the Green's matrix $\{\GG_{ij}\}$ for the model problem \eqref{model system}\eqref{model system bc} takes the following explicit form:
\begin{equation}\label{green matrix for oblique}
\GG_{ij}(x,y) = \frac{\delta_{ij}}{4\pi} \Big\{\G(x-y)-\G(x-\yst)-\frac{2 b^{(i)}_3}{3|x-\yst|}\TI(x,\yst)\Big\},
\end{equation}
where for each $i=1,2,3$,
\begin{equation}\label{theta in green matrix}
\TI(x,\yst) :=  \int_0^\infty\Bigg\{e^{a^{(i)}|x-\yst|s} \frac{\xi_3 + b^{(i)}_3 s}{\big[1+2\langle \bbi,\xi\rangle s+s^2\big]^{3/2}} \Bigg\}\,\dd s.
\end{equation}
In fact, later (in Lemma \ref{lemma: Theta kernel}) we shall check that $\TI$ is smooth in $(x,y)$. 

The above representation formulae \eqref{green matrix for oblique}\eqref{theta in green matrix} are the starting point of our subsequent estimates. Recall that the Green's matrix for the Dirichlet condition is 
	\begin{equation}
	\mathcal{G}^{\rm Dirichlet}_{ij}(x,y) = \frac{\delta_{ij}}{4\pi} \Big\{ \Gamma(x-y) - \Gamma(x-y^\star)\Big\},
	\end{equation}
which can be obtained as a special case of Eqs.\,\eqref{green matrix for oblique}\eqref{theta in green matrix}  by setting $b_3^{(i)}=0$ for each $i$ (hence the $\Theta^{(i)}$-term becoming zero). Thus, for the case of the regular oblique derivative condition in this paper, the major difference arises from the nontrivial conditions for $\partial u\slash \partial \mathbf{n}$ on the boundary. Our analyses in this paper cover more general cases than the Dirichlet condition, taken into account the $\Theta^{(i)}$-term ({\it cf.} also Remark~\ref{remark: regular oblique} below). Our notations for the integral kernels in this paper slightly differ from those in \cite{gt}.

\section{Basic Energy Estimates}

In this section we derive the energy estimates for the Navier--Stokes Eqs.\,\eqref{ns}\eqref{incompressible}, subject to the general Navier and kinematic boundary conditions in \eqref{navier bc}\eqref{kinematic bc}. Whenever the estimates are {\rm kinematic}, {\it i.e.}, valid pointwise in time, we  suppress the variable $t$ to simplify the presentation. Let us first fix some notations: for $a,b\in\R^3$, we write
\begin{equation*}
a\otimes b = \{a \otimes b\}_{ij} \in \mathfrak{gl}(3,\R), \qquad (a\otimes b)_{ij} := a^i b^j \text{ for } i,j\in\{1,2,3\};
\end{equation*}
and for $A, B \in \mathfrak{gl}(3,\R)$,  write
\begin{equation*}
A : B := {\rm Trace}\, (A B), \qquad  |A| := \sqrt{\sum_{i,j=1}^3 |A_{ij}|^2}.
\end{equation*}
We also need the following geometric quantities:
\begin{equation}
\two:=-\na\n: \G(T\Sigma)\times \G(T\Sigma)\map \G(T\Sigma^\perp)
\end{equation}
is the second fundamental form of $\Sigma$, and 
\begin{equation}
H_\Sigma := {\rm Trace}\,(\two)
\end{equation}
is the mean curvature of $\Sigma$. The metric on $\Sigma$ (with respect to which we are taking the trace) is the pullback of the Euclidean metric via the natural inclusion $\Sigma \emb \R^3$. We use $\G(T\Sigma)$ to denote the space of vector fields tangential to $\Sigma$, and use $\htwo$ to denote the 2-dimensional Hausdorff measure on $\Sigma$.

To begin with, let us take the gradient of Eq.\,\eqref{ns} and anti-symmetrise it. This gives us the {\em vorticity equation}:
\begin{equation}\label{vorticity eq}
\p_t \omega + u \cdot \na\omega = \nu \Delta \omega + \scal u\cdot \omega
\end{equation}
in $[0,T^\star[ \times \Omega$. In the sequel, for $\omega$ and $u$ to satisfy the Navier boundary condition \eqref{navier bc}, we understand \eqref{navier bc} in the sense of trace. In particular, let us impose the following, which shall be taken as part of the definition for the weak solutions to the system \eqref{ns}\eqref{incompressible}\eqref{navier bc}\eqref{kinematic bc}.
\begin{assumption}
Both the tangential and the normal traces of $\omega$ on $\Sigma = \p\Omega$ exist. The incompressibility condition $\na \cdot u = 0$ holds on $\Sigma$, also in the sense of trace.
\end{assumption}

Let us establish several energy estimates for the strong solutions. First, we note that the $L^2$ norm of $\na u$ can be bounded by the $L^2$ norm of $u$ and $\omega=\na\times u$, which can be shown by a direct integration by parts:
\begin{lemma}\label{lemma: div curl}
Let $u$ be a strong solution to Eqs.\,\eqref{ns}\eqref{incompressible}\eqref{navier bc}\eqref{kinematic bc} on $[0,T_\star[ \times \Omega$. Then, for all $t \in [0,T^\star[$,
\begin{equation}
\int_\Omega |\na u|^2\,\dd x \leq \|\two\|_{L^\infty(\Sigma)} \int_{\Omega}|u|^2\,\dd x + \int_\Omega |\omega|^2\,\dd x.
\end{equation}
\end{lemma}

\begin{proof}
See (3.4), p.728 in Chen--Qian \cite{cq}.
\end{proof}

The next result concerns the growth of {\em enstrophy}, namely the square of the $L^2$ norm of vorticity. One may  compare it with Lemma 2.6 in \cite{bb2} (recall that the vorticity stretching term $[Stretch]$ is defined in Eq.\,\eqref{stretch}):
\begin{lemma}\label{lemma: enstrophy}
Let $u$ be a strong solution to Eqs.\,\eqref{ns}\eqref{incompressible}\eqref{navier bc}\eqref{kinematic bc} on $[0,T_\star[ \times \Omega$. Then there exists a constant $c_0$ depending only on $\beta, \nu$ and $\|\two\|_{C^1(\Sigma)}$ such that for each $t \in [0,T^\star[$, we have
\begin{equation}
\frac{1}{2}\frac{\dd}{\dd t}\int_\Omega |\omega|^2\,\dd x + \frac{\nu}{2} \int_{\Omega} |\na \omega|^2\,\dd x - c_0 \int_\Sigma \Big(|\na u|^2+|u|^2\Big)\,\dd\htwo \leq [{\rm Stretch}].
\end{equation} 
\end{lemma}

\begin{proof}
We divide our arguments into four steps.

{\bf 1.} First, multiplying $\omega$ to the vorticity equation \eqref{vorticity eq}, we get
\begin{equation}\label{qqq}
\p_t(|\omega|^2) + u \cdot \na(|\omega|^2) - \nu \Delta(|\omega|^2) + 2 \nu |\na \omega|^2 = 2 \scal u: (\omega \otimes \omega).
\end{equation}
By Eq.\,\eqref{incompressible} and the divergence theorem,
\begin{equation*}
\int_{\Omega} u\cdot \na (|\omega|^2)\,\dd x = \int_\Sigma |\omega|^2 u \cdot \n\,\dd \htwo,
\end{equation*}
which vanishes due to the kinematic boundary condition.  On the other hand, by the divergence theorem again, we have
\begin{equation}\label{bdry term: bad}
\int_{\Omega} \Delta(|\omega|^2)\,\dd x =  \int_\Sigma \frac{\p |\omega|^2}{\p \n}\,\dd\htwo = 2\int_\Sigma \omega \cdot \frac{\p \omega}{\p\n}\,\dd\htwo,
\end{equation}
where $\p/\p\n := \n\cdot\na$. In view of Eq.\,\eqref{qqq} and the triangle inequality, it remains to establish
\begin{equation}\label{bad bdry trace}
\bigg|\int_{\Sigma} \omega \cdot \frac{\p \omega}{\p\n}\,\dd\htwo\bigg| \leq \frac{\nu}{2} \int_{\Omega} |\na \omega|^2\,\dd x + c_0 \int_\Sigma \Big(|\na u|^2+|u|^2\Big)\,\dd\htwo.
\end{equation}

\smallskip
{\bf 2.} To deal with the last term in \eqref{bdry term: bad}, let us utilise the Navier boundary condition \eqref{navier bc}. Take an arbitrary orthonormal frame $\{\p/\p x^i\}$ on $\R^3$, and suppose that $\tau = \p/\p x^k$ is a tangential vector field to $\Sigma$; then
\begin{align*}
0 &= \beta u^k + \sum_{i=1}^3\nu(\na_i u^k + \na_k u^i)\n^i\\
&= \beta u^k +  \sum_{i=1}^3\Big(\nu (-\na_k u^i + \na_i u^k)\n^i + 2\nu (\na_k u^i)\n^i\Big)\\
&= \beta u^k + \nu \sum_{i,l=1}^3\e^{ikl}\omega^l + 2\nu \na_k (\sum_{i=1}^3 u^i\n^i)-2\nu \sum_{i=1}^3 u^i \na_k \n^i
\end{align*}
for any $i\in\{1,2,3\}$. Thanks to $u \cdot \n = 0$ and the definition of the second fundamental form, we obtain an equivalent formulation of the Navier boundary condition as follows:
\begin{equation}\label{eq formulation for navier bc}
0 = \beta u \cdot \tau + \nu (\omega \times \n) \cdot \tau - 2\nu \two(u, \tau) \qquad \text{ on } \Sigma \text{ for each } \tau \in \G(T\Sigma).
\end{equation}
Moreover, note that if we decompose $\omega$ into tangential and normal components:
\begin{equation}
\omega := \omega^{\parallel} + \omega^{\perp} \qquad \text{ for } \omega^{\parallel} \in \G(T\Sigma),\, \omega^{\perp} \in \G(T\Sigma^\perp),
\end{equation}
then $\omega^{\parallel}$ can be pointwise controlled by $u$ and the geometry of $\Sigma$:
\begin{equation}\label{parallel estimate}
|\omega^{\parallel}| \leq \Big(\beta\nu^{-1} + 2 \|\two\|_{L^\infty(\Sigma)}\Big)|u|.
\end{equation}

\smallskip
{\bf 3.} Now let us estimate 
\begin{equation}\label{parallel and perp}
2\int_\Sigma \omega \cdot \frac{\p\omega}{\p\n}\,\dd\htwo = 2 \int_\Sigma \bigg\{\omega^{\parallel} \cdot \frac{\p\omega^{\parallel}}{\p\n} + \omega^\perp \cdot \frac{\p\omega^{\parallel}}{\p\n} + \omega^\parallel \cdot \frac{\p\omega^{\perp}}{\p\n} + \omega^\perp \cdot \frac{\p\omega^{\perp}}{\p\n} \bigg\}\,\dd\htwo
\end{equation}
in Eq.\,\eqref{bdry term: bad}. 
For the first two terms, let us use Eq.\,\eqref{eq formulation for navier bc} to derive that
\begin{equation}
\na \omega^{\parallel} = L(\na u, \na\two \star u, \two \star \na u),
\end{equation}
where the schematic tensor $L(X_1, X_2,\ldots)$ denotes a linear combination of $X_1, X_2, \ldots$ with coefficients depending only on $\beta, \nu$, and $X \star Y$ denotes a generic quadratic term in $X,Y$ with constant coefficients. Thus, we have the pointwise estimate
\begin{equation}\label{na omega parallel}
|\na \omega^\parallel| \leq C (|\na u| + |u|),
\end{equation}
where $C$ depends only on $\|\two\|_{C^1}, \beta$ and $\nu$. We can bound
\begin{align}\label{TT,NT}
\bigg|\int_\Sigma \bigg\{\omega^{\parallel} \cdot \frac{\p\omega^{\parallel}}{\p\n} + \omega^\perp \cdot \frac{\p\omega^{\parallel}}{\p\n}\bigg\}\,\dd\htwo\bigg| &\leq C\int_\Sigma \Big(|\omega||\na u| + |\omega||u|\Big)\,\dd\htwo\nonumber\\
&\leq C\bigg\{2\int_\Sigma|\na u|^2\,\dd\htwo +\int_\Sigma |u|^2 \,\dd\htwo\bigg\},
\end{align}
using Eq.\,\eqref{na omega parallel} and Cauchy--Schwarz, with the constant $C=C(\|\two\|_{C^1}, \beta, \nu, \Sigma)$. 

For the third term, notice that
\begin{equation*}
\int_\Sigma \omega^\parallel \cdot \frac{\p\omega^{\perp}}{\p\n} \,\dd\htwo = -\int_\Sigma \omega^\perp \cdot \frac{\p\omega^{\parallel}}{\p\n} \,\dd\htwo + \int_\Sigma \n\cdot\na (\omega^{\parallel}\cdot\omega^\perp)\,\dd\htwo=-\int_\Sigma \omega^\perp \cdot \frac{\p\omega^{\parallel}}{\p\n} \,\dd\htwo.
\end{equation*}
Just as above, we get
\begin{equation}\label{TN}
\bigg|\int_\Sigma \omega^\parallel \cdot \frac{\p\omega^{\perp}}{\p\n} \,\dd\htwo \bigg|\leq C\int_\Sigma \Big(|\na u|^2 + |u|^2 \Big)\,\dd\htwo.
\end{equation}

\smallskip
{\bf 4.} To control the remaining term $\int_\Sigma \omega^\perp\cdot(\p\omega^\perp/\p\n)\,\dd\htwo$, let us first establish a simple {\em claim}: For any vertical vector field $\eta \in \G(T\Sigma^\perp)$, there holds
\begin{equation}
\frac{1}{2}\n \cdot\na (|\eta|^2) - (\na \cdot \eta)(\eta \cdot \n) = H_\Sigma(|\eta|^2).
\end{equation}
Indeed, write $\eta=\phi\n$ for some scalar function $\phi: \Sigma \map \R$. Then
\begin{align*}
\frac{1}{2}\n \cdot\na (|\eta|^2) - (\na \cdot \eta)(\eta \cdot \n)&=\sum_{i,j=1}^3 \Big(\eta^j \n^i\na_i \eta^j  - (\na_i \eta^i) (\eta^j\n^j)\Big)\\
&= \phi^2\sum_{i,j=1}^3 (\n^i\n^j\na_i\n^j-\na_i\n^i) = H_\Sigma\phi^2,
\end{align*}
where the last equality follows from $|\n|=1$ and the definition of mean curvature. As a side remark, this {\em claim} gives a geometric interpretation to the boundary term in the case of the ``slip-type'' boundary condition $\omega \times \n =0$ as in Lemma 2.6, \cite{bb2}.

In the above {\em claim} let us take $\eta=\omega^\perp$. Thanks to the incompressibility of $\omega$, we have $\na \cdot \omega^{\parallel} = -\na \cdot \omega^\perp$; thus, 
\begin{equation}
\omega^\perp \cdot \frac{\p\omega^{\perp}}{\p\n} = - (\na\cdot\omega^{\parallel}) |\omega^\perp| + H_\Sigma |\omega^\perp|^2.
\end{equation}
Therefore, using Eq.\,\eqref{na omega parallel} again and arguing as in \eqref{TT,NT}, one obtains
\begin{equation}\label{NN}
\bigg|\int_\Sigma \omega^\perp \cdot \frac{\p\omega^{\perp}}{\p\n} \,\dd\htwo \bigg|\leq C\int_\Sigma \Big(|\na u|^2 + |u|^2 \Big)\,\dd\htwo.
\end{equation}
Finally, we put together Eqs.\,\eqref{parallel and perp}\eqref{TT,NT}\eqref{TN}\eqref{NN} to complete the proof.  \end{proof}




The lemma below justifies the energy inequality \eqref{energy ineq} in the definition of weak solutions:

\begin{lemma}\label{lemma: energy ineq}
Let $u$ be a strong solution to  Eqs.\,\eqref{ns}\eqref{incompressible}\eqref{navier bc}\eqref{kinematic bc} on $[0,T_\star[ \times \Omega$. There exists a constant $c_1>0$ depending only on $\beta$, $\nu$ and $\|\Sigma\|_{L^\infty(\Omega)}$ such that 
\begin{equation}
\frac{1}{2}\frac{\dd}{\dd t}\int_\Omega |u|^2\,\dd x + \nu \int_\Omega |\na u |^2\,\dd x - c_1 \int_\Sigma |u|^2\,\dd \htwo \leq 0
\end{equation}
for each $t\in]0,T^\star[$.
\end{lemma}

\begin{proof}
This follows from standard energy estimates. Multiplying $u$ to the Navier--Stokes equations \eqref{ns}\eqref{incompressible} and integration by parts, we have
\begin{align}\label{...}
\frac{1}{2}\frac{\dd}{\dd t} \int_\Omega |u|^2\,\dd x + \nu\int_\Omega |\na u|^2\,\dd x -\nu\int_\Sigma u\cdot\frac{\p u}{\p\n}\,\dd \htwo = 0.
\end{align}
To estimate the last term, let $\{\p/\p x^i\}_{i=1}^3$ be an arbitrary local orthonormal frame on $\R^3$; then
\begin{align*}
u \cdot \frac{\p u}{\p \n} &= \sum_{i,j=1}^3 u^i\n^j\na_j u^i\\
&= \sum_{i,j=1}^3 \Big(u^i\n^j(\na_j u^i - \na_i u^j) + u^i\n^j\na_i u^j\Big)\\
&= \sum_{i,j,k=1}^3 \Big(\e^{kji}u^i\n^j\omega^k + u^i \na_i (u^j\n^j) - u^i u^j\na_i \n^j\Big)\\
&= u \cdot (\omega \times \n) + u\cdot \na (u\cdot\n) +\two(u,u).
\end{align*}
In view of the incompressibility of $u$ and that $\omega \times \n = \omega^\parallel\times\n$, we have
\begin{equation}
\bigg|\int_\Sigma u \cdot \frac{\p u}{\p\n}\,\dd\htwo\bigg| \leq \int_\Sigma \Big(|u| |\omega^\parallel| + \|\two\|_{L^\infty(\Sigma)}|u|^2 \Big)\,\dd\htwo.
\end{equation}
But $|\omega^\parallel|$ can be estimated by $|u|$ as in Eq.\,\eqref{parallel estimate}; by \eqref{...}, we may thus take $$c_1:=\beta+3\nu\|\two\|_{L^\infty(\Sigma)}$$ to complete the proof.    \end{proof}

Several bounds can be deduced immediately from Lemmas \ref{lemma: div curl}, \ref{lemma: enstrophy} and \ref{lemma: energy ineq}. First, by the trace inequality
\begin{equation*}
c_1 \int_{\Sigma}|u|^2\,\dd\htwo \leq \frac{\nu}{2}\int_{\Omega} |\na u|^2\,\dd x + {c_2}\int_{\Omega} |u|^2\,\dd x,
\end{equation*}
Lemma \ref{lemma: energy ineq} implies 
\begin{equation}\label{energy x}
\frac{\dd}{\dd t}\int_\Omega |u|^2\,\dd x + \nu \int_\Omega |\na u|^2\,\dd x \leq 2c_2 \int_\Omega |u|^2\,\dd x,
\end{equation}
where $c_2$ depends on $c_1, \Omega$ and $\nu^{-1}$. Then, thanks to Gr\"{o}nwall's lemma, one has
\begin{equation}\label{energy xx}
\|u(t,\cdot)\|_{L^2(\Omega)} \leq \|u_0\|_{L^2(\Omega)} e^{c_2 t} \qquad \text{ for each } t\in [0,T_\star[.
\end{equation}
Thus,
\begin{equation}\label{energy xxx}
\|\na u(t,\cdot)\|_{L^2(\Omega)} \leq \sqrt{{2 c_2}/{\nu}} \|u_0\|_{L^2(\Omega)} e^{c_2 t} \qquad \text{ for each } t\in [0,T_\star[.
\end{equation}
Next, applying again the trace inequality to Lemma \ref{lemma: enstrophy} yields that, for any given $\delta>0$,
\begin{align*}
&\frac{1}{2}\frac{\dd}{\dd t}\int_\Omega |\omega|^2\,\dd x + \frac{\nu}{2}\int_\Omega |\na \omega|^2\,\dd x \\
&\qquad\qquad\leq [{\rm Stretch}] + \delta\int_\Omega |\na \na u|^2\,\dd x + \frac{c_3}{\delta}\int_\Omega |\na u|^2\,\dd x + c_3 \int_\Omega |\na u|^2\,\dd x + c_3 \int_\Omega |u|^2\,\dd x.
\end{align*}
Here $c_3$ depends on $c_0$ and $\Omega$. By \eqref{energy xx}\eqref{energy xxx}, the last three terms on the right-hand side are bounded by a constant $c_4=C(c_3, c_2, \nu, \|u_0\|_{L^2(\Omega)}, T_\star,\delta)$. Moreover, we have
\begin{lemma}\label{lemma: na na u}
Let $u$ be a strong solution to  Eqs.\,\eqref{ns}\eqref{incompressible}\eqref{navier bc}\eqref{kinematic bc} on $[0,T_\star[ \times \Omega$. There exists $c_5$ depending only on $\Omega$ such that 
\begin{equation}
\int_\Omega |\na \na u|^2\,\dd x \leq c_5 \Big(\int_{\Omega} |\na \omega|^2\,\dd x + \int_\Omega |\omega|^2\,\dd x + \int_\Omega |u|^2\,\dd x\Big).
\end{equation}
\end{lemma}
\begin{proof}
	This is a weaker result than Theorem 3.3, p.729 in Chen--Qian \cite{cq}.  \end{proof}

Therefore, choosing $\delta:=\nu/(4 c_5)$ and invoking once more \eqref{energy xx}\eqref{energy xxx}, one may conclude:
\begin{theorem}[Energy Estimate]\label{thm: energy estimate} Let $u$ be a strong solution to Eqs.\,\eqref{ns}\eqref{incompressible}\eqref{navier bc}\eqref{kinematic bc} on $[0,T^\star[ \times \Omega$. There is a constant $M$ depending on $\Omega, \beta, \|\two\|_{C^1(\Sigma)}, \nu, \|u_0\|_{L^2(\Omega)}$ and $T_\star$, such that \begin{equation}\label{enstrophy xx}
\frac{1}{2}\frac{\dd}{\dd t}\int_\Omega |\omega|^2\,\dd x + \frac{\nu}{4}\int_\Omega |\na \omega|^2\,\dd x \leq  [{\rm Stretch}] + M.
\end{equation}
The vorticity stretching term $[Stretch]$ is defined in Eq.\,\eqref{stretch}. Moreover, the supremum of $\|u(t,\cdot)\|_{W^{1,2}(\Omega)}$ is bounded on $[0,T_\star[$ by \eqref{energy x}\eqref{energy xx}.
\end{theorem}

\section{Boundary Regularity and Alignment of Vorticity Up to the Boundary}

In this section let us prove Theorem \ref{thm: new}. It is a generalisation of Theorem \ref{thm: main} (see Sect.\,5), with the more general diagonal oblique derivative boundary condition \eqref{homog bc} considered on arbitrary regular curvilinear domains rather than the Navier and kinematic boundary conditions \eqref{navier bc}\eqref{kinematic bc} on round balls, half-spaces and right cylinders. 
\begin{theorem}\label{thm: new}
Let $\Omega \subset \R^3$ be a sufficiently regular domain. Let $u$ be a weak solution to the Navier--Stokes equations \eqref{ns}\eqref{incompressible} on $[0,T_\star[ \times \Omega$  with the regular oblique derivative boundary condition \eqref{homog bc}. Assume that the energy estimate in Theorem \ref{thm: energy estimate} is valid for strong solutions. Then, under the assumptions of Theorem \ref{thm: main}, {\it i.e.}, if the vorticity turning angle $\theta$ satisfies
\begin{equation}\label{assumption'}
|\sin \theta(t;x,y)| \leq \rho \sqrt{|x-y|} \qquad \text{ for all } t \in [0,T^\star[,\, x,y\in\overline{\Omega},
\end{equation}
for some $\rho >0$, then $u$ is also a strong solution on $[0,T^\star[\times\Omega$. 
\end{theorem}
We emphasise that in the assumption \eqref{assumption'} above, the inequality holds for $x,y \in \overline{\Omega} = \Omega \cup \Sigma$; that is, we require that the vorticity is coherently aligned {\em up to the boundary}.

\begin{remark}\label{remark: regular oblique}The theorem above also applies to the Dirichlet condition $u \equiv 0$ on $\partial\Omega$.  First of all, the discussions in Sect.\,\ref{Green's Matrices} remain to be valid for the Dirichlet condition; in particular, Eqs.\,\eqref{Poisson eq}\eqref{homog bc} form an elliptic system of the Petrovsky type. In addition, as remarked at the end of Sect.\,2, the Green's matrix for the Dirichlet condition is the special case of Eqs.\,\eqref{green matrix for oblique}\eqref{theta in green matrix}  with vanishing $\Theta^{(i)}$-terms. As a consequence, all the arguments in the current section will carry through for the Dirichlet condition, with the modification that $J_{213}$ and its derivatives are all equal to zero ({\it i.e.}, Sect.\,\ref{subsec J213} holds trivially). It suggests that the geometric boundary regularity criterion in this paper may persist under the (formal) limit $\beta \uparrow +\infty$ of the Navier boundary condition \eqref{navier bc}.
\end{remark}


To prove Theorem \ref{thm: new}, in Sect.\,4.1 we first localise the problem to small coordinates charts in the interior or near the boundary. The key is to  estimate the vortex stretching term $[Stretch]$, which is carried out in Sects.\,4.2--4.10. Finally, we conclude the proof in Sect.\,4.11, thanks to the preliminary energy estimates obtained in Sect.\,3.

Let us also comment on the general strategy for the proof. It is based on the following continuation argument. From the definition of weak and strong solutions, we know that 
\begin{equation}
\limsup_{t \uparrow T} \int_\Omega |\omega(t)|^2\,\dd x =\infty,
\end{equation} 
is a {\em breakdown criterion} for strong solutions. That is, a weak solution $u$ on $[0,T[$ cannot be strong beyond the time $T$ if the above quantity blows up. Therefore, we assume that $u$ is a strong solution on $[0,T[$ for some $T \leq T^\star$. Utilising the energy estimate in Theorem \ref{thm: energy estimate} and the bound for $[Stretch]$ in the current section, we prove that the above blowup does not occur. Thus, $u$ is strong on $[0,T + \delta]$ for some $\delta >0$, which gives us the contradiction. Therefore, $u$ is strong all the way up to $T^\star$.

\subsection{Localisation}
We adopt Solonnikov's method of localisation in the construction of Green's matrices; see p.150 in \cite{s1} and p.609 in \cite{bb2}. For the convenience of the readers, let us briefly summarise the construction in four steps below:

{\bf 1.} There exists a finite family of open cover for $\overline{\Omega}$, written as
\begin{equation}
\big\{U_a\big\}_{a\in \I} \sqcup \big\{U_b\big\}_{b\in \B},
\end{equation}
where $U_a \cap \Sigma = \emptyset$ for each $a\in\I$, and $U_b \cap \Sigma \neq \emptyset$ for each $b \in \B$. Each $U_a$ is known as an interior chart, and each $U_b$ as a boundary chart. 

{\bf 2.} Each interior chart is a cube: there exists $d_1>0$ (independent of $a\in\I$) such that
\begin{equation}
U_a = \big\{(x^1, x^2, x^3)\in\R^3: |x^i-x_a^i| \leq d_1 \big\} \qquad \text{ for some } x_a \in \R^3,
\end{equation}
which also satisfies
\begin{equation}
{\rm dist} \,(U_a, \Sigma)  \geq d_1.
\end{equation}

{\bf 3.} In each boundary chart $U_b$, we can find a boundary point $x_b \in \Sigma$, a local Euclidean coordinate system $\{z_b^1, z_b^2, z_b^3\}$, and a $C^2$ map $\mathcal{F}_b: [0,d_2]^2 \map \R$ such that
\begin{equation}
|z_b^1|, |z_b^2| \leq d_2, \qquad 0\leq z_b^3 - \mathcal{F}_b(z_b^1, z_b^2) \leq 2d_2
\end{equation}
for some constant $d_2>0$ independent of $b \in \B$, and that the portion of the boundary $\Sigma \cap U_b$ in $z_b$ coordinates is the graph of $\mathcal{F}_b$. 

{\bf 4.} Let $\{\chi_a\}_{a\in\I}\cup\{\chi_b\}_{b\in\B}$ be a $C^\infty$ partition of unity subordinate to the cover in Step 1. That is, $0\leq\chi_c\leq 1$, $\chi_c \in C^\infty_c(\overline{\Omega})$, $\sum_{c\in \I \cup \B}\chi_c (x)= 1$ for each $x\in\overline{\Omega}$ and ${\rm spt}\,(\chi_c)\in U_c$ for each $c\in \I$ or $\B$.

With the help of the above steps, we can now localise the Green's matrices. Indeed, in Step 3 above let us further introduce the  notations:
\begin{equation}
z_b := \OO_b(x-x_b) \qquad \text{ for } \OO_b \in SO(3),
\end{equation}
\begin{equation}
\big((z')_b^1, (z')_b^2, (z')_b^3\big) := \big(z_b^1, z_b^2, z_b^3-\mathcal{F}_b(z_b^1, z_b^2)\big)\equiv \widetilde{\mathcal{F}}_b(z_b)
\end{equation}
and
\begin{equation}\label{Tb}
T_b(x):=\widetilde{\mathcal{F}}_b\circ\OO_b(x-x_b).
\end{equation}
That is, $\OO_b$ is the rotation of Euclidean coordinates, and $\widetilde{\mathcal{F}}_b \in C^2(U_b; [0,d_2]^2\times[0,2d_2])$ is the boundary straightening map, which satisfies
\begin{equation}
T_b(\Sigma\cap U_b) \subset \{(z')_b^3=0\}.
\end{equation}
Then, setting
\begin{equation}\label{d_3}
d_3:=\frac{\min\{d_1,d_2\}}{4},
\end{equation} 
we can compute $u(x)$ from the following explicit integral formula (comparing with Eq.\,(29) in Beir\~{a}o da Veiga--Berselli \cite{bb2}, p.610). 
\begin{lemma}\label{lemma: representation formula}
Let $u$ be a strong solution to Eqs.\,\eqref{Poisson eq}\eqref{homog bc}. Fix a cut-off function $\zeta \in C^\infty_c(\R)$ such that $0 \leq \zeta\leq 1$, $\zeta$ is non-increasing on $\R$, $\zeta \equiv 1$ on $[0,1/4]$, $\zeta \equiv 0$ on $[3/4, \infty[$ and $\|\zeta'\|_{C^0(\R)} \leq 4$. Then we have
\begin{align}\label{u}
u^i(x) &= \sum_{j=1}^3\sum_{a\in\I}\int_{\Omega}  \chi_a(y) \bigg\{\frac{\delta_{ij}}{4\pi|x-y|} (\na\times\omega)^j(y)\bigg\} \zeta\bigg(\frac{|x-y|}{d_3}\bigg)\,\dd y \nonumber\\
&\qquad+\sum_{j=1}^3\sum_{b\in\B}\int_\Omega\chi_b(y) \bigg\{\frac{\delta_{ij}}{4\pi}\bigg[\frac{1}{|T_bx-T_by|} - \frac{1}{|T_bx-(T_by)^\star|}\times\nonumber\\
&\qquad\qquad\qquad \times\Big(1+\frac{2b_3^{(i)}}{3}\TI\big(T_bx, (T_by)^\star\big) \Big)\bigg]
(\na\times\omega)^j(y)\bigg\}\zeta\bigg(\frac{|T_bx-T_by|}{d_3}\bigg)\,\dd y\nonumber\\
&\qquad+\sum_{j=1}^3\int_\Omega {\GG^\good}_{ij}(x,y) (\na\times\omega)^j(y)\,\dd y\nonumber\\
&=:J_1(x) + J_2(x) + J_3(x),
\end{align}
where $\GG^\good$ satisfies the estimate in Eq.\,\eqref{good}, and $\TI$ is given by Eq.\,\eqref{theta in green matrix}.
\end{lemma}

\begin{proof}
As Eqs.\,\eqref{Poisson eq}\eqref{homog bc} are an elliptic system of Petrovsky type, by Lemma \ref{lemma: solonnikov} we can find one single Green's matrix $\GG$ such that 
	\begin{align}
	u^i(x) &= \sum_{a\in\I}\sum_{j=1}^3\int_\Omega \GG_{ij}(x,y)\chi_a(y)(\na\times\omega)^j(y)\,\dd y + \sum_{b\in\B}\sum_{j=1}^3\int_\Omega \GG_{ij}(x,y)\chi_b(y)(\na\times\omega)^j(y)\,\dd y\nonumber\\
	&= u^i_{\rm int} (x) + u^i_{\rm bdry} (x),
	\end{align}
where $\{\chi_a\}_{a\in\I}\sqcup\{\chi_b\}_{b\in\B}$ is the aforementioned partition-of-unity.

For $u^i_{\rm int} (x)$, let us decompose each of its summands as
\begin{align}
u^i_{\rm int, \, near}(x) +u^i_{\rm int, \, far}(x) = &\int_\Omega \GG_{ij}(x,y)\chi_a(y)(\na\times\omega)^j(y)\bigg\{\zeta\bigg(\frac{|x-y|}{d_3}\bigg)\bigg\}\,\dd y \nonumber\\
 +\,&  \int_\Omega \GG_{ij}(x,y)\chi_a(y)(\na\times\omega)^j(y)\bigg\{1- \zeta\bigg(\frac{|x-y|}{d_3}\bigg)\bigg\}\,\dd y.
\end{align}
The non-zero contribution to $u^i_{\rm int, \, near}(x)$ comes from $\{y\in U_a: |y-x|\leq 3d_3/4\}$, which is uniformly away from the boundary $\Sigma$. Thus
\begin{equation}\label{a}
\GG_{ij}(x,y)\1_{\{y\in U_a: |y-x|\leq 3d_3/4\}} = \frac{\delta_{ij}}{4\pi}\G(x,y) + \GG^{\good}_{ij}(x,y),
\end{equation}
where the leading term $\frac{\delta_{ij}}{4\pi}\G(x,y)$ is the Green's matrix on $\R^3$, and the error term $\GG^{\good}$ satisfies \eqref{good} (the explicit form of $\GG^\good$ may differ from line to line, though). On the other hand, the non-zero contribution to $u^i_{\rm int, \, far}(x)$ comes only from $\{y\in U_a: |y-x| > d_3/4\}$, but the Green's matrix $\GG_{ij}$ is smooth away from the diagonal $\{x=y\}\subset \R ^3 \times \R^3$. That is, 
\begin{equation}\label{b}
\GG_{ij}(x,y)\1_{\{y\in U_a: |y-x| > d_3/4\}} = \GG^\good_{ij}(x,y).
\end{equation}

For the boundary term $u^i_{\rm bdry}(x)$, we apply the boundary-straightening map $T_b$ in each boundary chart; {\it cf.} Eq.\,\eqref{Tb}. Indeed, for each $x\in U_b$, $b\in\B$, arguments analogous to those for the $u^i_{\rm int}(x)$ term show that
\begin{align}\label{c}
u_{\rm bdry}^i(x) &= \sum_{b\in\B}\sum_{j=1}^3\int_\Omega \GG_{ij}(x,y) \chi_b(y) (\na\times\omega)^j(y)  \zeta\bigg(\frac{|x-y|}{d_3}\bigg)\,\dd y  \nonumber\\
&\qquad\qquad\qquad+ \sum_{j=1}^3\int_\Omega\GG^\good_{ij}(x,y) (\na\times\omega)^j(y)\,\dd y.
\end{align}

We further {\em claim} that
\begin{align}\label{d}
&\int_\Omega \GG_{ij}(x,y) \chi_b(y) (\na\times\omega)^j(y)  \zeta\bigg(\frac{|x-y|}{d_3}\bigg)\,\dd y \nonumber\\
 =\,& \int_\Omega \GG_{ij}(T_bx,T_by) \chi_b(y) (\na\times\omega)^j(y)  \zeta\bigg(\frac{|T_bx-T_by|}{d_3}\bigg)\,\dd y + \int_\Omega\GG^\good_{ij}(x,y) (\na\times\omega)^j(y).
\end{align}
Indeed, by the definition of $T_b$ we have
\begin{equation}\label{DT matrix}
\na T_b (x) = \begin{bmatrix}
1&0&0\\
0&1&0\\
-\na_1\mathcal{F}_b&-\na_2\mathcal{F}_b&1
\end{bmatrix} \cdot \OO_b(x-x_b), \qquad \text{ where } \OO_b \in O(3).
\end{equation} 
So  $|\det (\na T_b)| = 1$; thus $\na T_b(\cdot) \in O(3)$ modulo a translation in $\R^3$. It means that the boundary-straightening map $T_b$ is almost a Euclidean isometry. Now, Taylor expansion gives us
\begin{equation}
|T_b x-T_b y| = |x-y| + \smallo (|x-y|)
\end{equation}
and
\begin{equation}
|\GG_{ij} (T_b x, T_by) -\GG_{ij}(x,y)| = \smallo(|x-y|),
\end{equation}
where $\smallo$ is the usual ``small-o'' notation in the limit of $|x-y|\map 0$. These higher order terms contribute to $\GG^\good$, as they cancel the singularities in the denominator of $\GG_{ij}$; see Lemma \ref{lemma: solonnikov}. Therefore, the {\em claim}  \eqref{d} follows.

Finally,  in the boundary chart $U_b$, the boundary condition  pulled back by $T_b$ is in the form of \eqref{model system bc}, which is the oblique derivative boundary condition on the half-space. Thus, choosing the local coordinate frame $\{x^1,x^2,x^3\}$ such that $\p/\p x^3 = \n$, we have
\begin{equation}\label{e}
\GG_{ij}(x,y) \1_{\{(x,y) \in U_b \times U_b\}} = \frac{\delta_{ij}}{4\pi} \bigg\{\G(T_bx-T_by)-\G(T_bx-(T_by)^\star) - \frac{2b_3^{(i)}}{3} \TI\big(T_bx, (T_b y)^\star\big) \bigg\}
\end{equation}
by \eqref{green matrix for oblique} (also see Sect.\,6.7 in Gilbarg--Trudinger \cite{gt}).  Eqs.\,\eqref{a}\eqref{b}\eqref{c}\eqref{d} and \eqref{e} together complete the proof.  \end{proof}

In the proof we have deduced the following identity:
\begin{equation}\label{DT}
\{\na T_b\}^i_j(x) \equiv \na_i (T_b x)^j = \sum_{k=1}^3 \OO^k_j\big(\delta^i_k - \delta^i_3\na_k\mathcal{F}_b\big)(x-x_b)\qquad \text{ for each } i,j \in \{1,2,3\},
\end{equation}
where $\OO^k_j \in O(3)$; see Eq.\,\eqref{DT matrix}. It will be repeatedly used in the subsequent development. 

\subsection{Potential Estimates for the Vortex Stretching Term}
In the following nine subsections we shall estimate the term
\begin{equation*}
[{\rm Stretch}] := \Big|\int_\Omega \scal u(t,x):\omega(t,x)\otimes \omega(t,x)\,\dd x \Big|
\end{equation*}
using the representation formula for $u$ in Lemma \ref{lemma: representation formula}; recall that $\scal u =(\na u + \na^\top u)/2$. To this end, we first need the expressions for $\na J_i:\omega\otimes\omega$, $i=1,2,3$. The major novelty and difficulty of the current work comes from the $J_2$ term, due to the non-triviality of the boundary conditions.

Before further development, let us introduce a notation used throughout the paper:
\begin{equation*}
\widehat{a} := \frac{a}{|a|}\qquad \text{ for } a \in \R^3.
\end{equation*}
Also, in what follows let us write $\na_{y,j}=\na_j$ for $\p/\p y^j$, and $\na_{x,k}$ for $\p/\p x^k$. Furthermore, $\e^{klj}$ denotes the Levi-Civita tensor which equals to $1$ if $(klj)$ is an even permutation of $(123)$, to $-1$ if $(klj)$ is an odd permutation of $(123)$, and to $0$ if there are repeated indices in $\{k,l,j\}$

\subsection{Estimates for $J_2$: Preliminaries}

Let us first integrate by parts to re-write the $J_2$ term. It suffices to bound $J_2$ in each fixed $U_b$ for $b\in\B$. With a slight abuse of notations, let us denote
\begin{equation}\label{j2}
J_2^i(x) := \sum_j \int_\Omega \bigg\{ \chi_b(y) \GG_{ij}(x,y) (\na \times \omega)^j(y)\zeta\bigg(\frac{|T_bx-T_by|}{d_3}\bigg)\bigg\}\,\dd y,
\end{equation}
where
\begin{equation}
\GG_{ij}(T_bx, T_by) = \frac{\delta_{ij}}{4\pi} \bigg\{\frac{1}{|T_b x-T_b y|} - \frac{1}{|T_bx - (T_by)^\star|}\Big(1+\frac{2b_3^{(i)}}{3}\TI \big(T_b x, (T_by)^\star\big) \Big) \bigg\},
\end{equation}
and $d_4$ is chosen to be the minimum of $d_3/2$ and the maximal width of the tubular neighbourhood  of $\Sigma = \p\Omega$ such that the nearest point projection onto $\Sigma$ is a homotopy retract. Also, to simplify the notations, we fix $b \in \B$ and drop the subscripts $_b$ from now on.

Recall that
$$
(\na \times \omega) (y) = \sum_{k,l,j=1}^3 \e^{klj}\na_k\omega^l \frac{\p}{\p y^j},\qquad (\omega\times\n)(y) = \sum_{k,l,j=1}^3 \e^{klj} \omega^k\n^l\frac{\p}{\p y^j},
$$
where $\e^{klj}$ is the Levi-Civita symbol. Thus, integrate by parts and use the Stokes' theorem, we obtain
\begin{align}\label{j2: three terms}
J_2^i(x) &= -\sum_{kjl}\e^{klj}\int_\Omega \chi(y)\zeta\bigg(\frac{|Tx-Ty|}{d_4}\bigg) \cdot \na_k\Big(\GG_{ij}(Tx,Ty)\Big)\omega^l(y)\,\dd y\nonumber\\
&\qquad -\sum_{kjl}\e^{klj} \int_\Omega \na_k\bigg[\chi(y)\zeta\bigg(\frac{|Tx-Ty|}{d_4} \bigg) \bigg] \GG_{ij}(Tx,Ty) \omega^l(y)\,\dd y  \nonumber\\
&\qquad -\sum_j\int_{\Sigma=\p\Omega} \chi(y)\zeta\bigg(\frac{|Tx-Ty|}{d_4}\bigg) \GG_{ij}(Tx,Ty) (\omega\times\n)^j(y)\,\dd\htwo(y) \nonumber\\
&=: J_{21}^i(x) + J_{22}^i(x) + J_{23}^i(x).
\end{align}
Here $\na_k = \p/\p y^k$ and $\htwo$ is the 2-dimensional Hausdorff measure on $\Sigma$ obtained from the inclusion $\Sigma \emb \R^3$.

In the subsequent six subsections (Sects.\,4.4--4.9), we estimate the terms $J_{2j}$, $j=1,2,3$ one by one.

\subsection{Decomposition of $J_{21}$ into Three Terms} 
Let us introduce the symbol
\begin{equation}
\sigma_j := \begin{cases}
1 \,\,\,\,\qquad \text{ if } i=1 \text{ or } 2,\\
-1 \qquad \text{ if } i = 3,
\end{cases}
\end{equation}
and adopt the convention $\gamma,\eta \in \{1,2\}$; $i,j,k,l,p,q\ldots\in\{1,2,3\}$. Then, $J_{21}$ can be further decomposed into three terms:
\begin{lemma}\label{lemma J21}
$J_{21}$ can be written as follows:
\begin{align}\label{J21}
[J_{21}(x)]^i &= \sum_{klp\gamma}\frac{\e^{kli}}{4\pi}\bigg\{\int_\Omega \chi(y)  \zeta\bigg(\frac{|Tx-Ty|}{d_4}\bigg)\bigg[-2\frac{(Tx-Ty)^p}{|Tx-Ty|^3}\Big(\OO^k_p -\delta^k_3\na_\gamma\F(y-x_b) \OO^\gamma_p\Big)
\nonumber\\
&\qquad\qquad +2 \frac{(Tx-\tyst)^p}{|Tx-\tyst|^3} \sigma_p\Big(\OO^k_p -\delta^k_3\na_\gamma\F(y-x_b) \OO^\gamma_p\Big)\nonumber\\
&\qquad\qquad + \frac{2b^{(i)}_3}{3|Tx-\tyst|} \na_k\big[\TI\big(Tx,\tyst\big)\big] \bigg] \omega^l(y) \bigg\}\,\dd y\nonumber\\
&=: [J_{211}(x)]^i +[J_{212}(x)]^i+[J_{213}(x)]^i.
\end{align}
\end{lemma}

Here and in the sequel, $\F:=\F_b$ as in Step 3 in Sect.\,4.1, and $x_b$ is the centre of the boundary chart $U_b$.

\begin{proof}
	It follows from a direct computation for $\na_k\GG_{ij}$. Note that 
	\begin{equation*}
	\na_k  \Big(\frac{1}{|Tx-Ty|}\Big) = \sum_{p=1}^3 \frac{-2(Tx-Ty)^p \na_p(Ty)^k}{|Tx-Ty|^3},
	\end{equation*}
where $\na_p (Ty)^k =\sum_q (\na T)_p^q\na_q y^k=(\na T)_p^k$. Thus,
\begin{equation}
\na_k \Big(\frac{1}{|Tx-Ty|}\Big) = -2\sum_{p=1}^3\sum_{\gamma=1}^2 \frac{(Tx-Ty)^p}{|Tx-Ty|^3} \Big\{\OO^k_p - \delta^k_3 \na_\gamma\F(y-x_b)\OO^\gamma_p \Big\}.
\end{equation}
Analogously, we have 
\begin{equation}
 \na_k \Big(\frac{1}{|Tx-\tyst|}\Big) = -2\sum_{p=1}^3\sum_{\gamma=1}^2 \frac{(Tx-\tyst)^p \sigma_p}{|Tx-\tyst|^3} \Big\{\OO^k_p - \delta^k_3 \na_\gamma\F(y-x_b)\OO^\gamma_p \Big\}.
\end{equation}
Hence, the assertion follows from the explicit formula for $\GG_{ij}$ in Eq.\,\eqref{green matrix for oblique}.  \end{proof}

In what follows we compute the vortex stretching terms involving $J_{21k}$, $k=1,2,3$ in order.

\subsection{Estimates for $J_{211}$} For this term, one has
\begin{align}
\na_j [J_{211}]^i(x)&= \sum_{klpq}\sum_{\gamma\eta} \frac{\e^{kli}}{2\pi} \Bigg\{ \int_\Omega \frac{2}{d_4} \chi(y) \zeta'\bigg(\frac{|Tx-Ty|}{d_4}\bigg) \frac{(Tx-Ty)^q}{|Tx-Ty|}\Big(\OO^j_q - \delta^j_3\na_\eta\F(x-x_b)\OO^\eta_q\Big)\times\nonumber\\
&\quad\times\frac{(Tx-Ty)^p}{|Tx-Ty|^3}\Big(\OO^k_p - \delta^k_3\na_\gamma\F(y-x_b)\OO^\gamma_p\Big) \omega^l(y) \,\dd y\nonumber\\
&\quad+ \int_\Omega \chi(y)\zeta\bigg(\frac{|Tx-Ty|}{d_4}\bigg)\na_{x,j}\bigg[\frac{(Tx-Ty)^p}{|Tx-Ty|^3}\Big(\OO^k_p - \delta^k_3\na_\gamma\F(y-x_b)\OO^\gamma_p\Big)\bigg] \omega^l(y)\,\dd y \Bigg\}\nonumber\\
&=: K_1(x) + K_2(x).
\end{align}
In the sequel let us simply the notations by setting
\begin{equation}\label{Xi}
\Xi(z)^i_j :=  \OO^i_j -\sum_{\gamma=1}^2 \delta^i_3\na_\gamma\F(z-x_c)\OO^\gamma_q\qquad \text{ for } z\in U_c,\,c\in\I\sqcup\B.
\end{equation}
Then, as $\Omega$ is a $C^2$ bounded domain,
\begin{equation}
\|\Xi\|_{C^0(U_c)} \leq 2+ \|\F\|_{\rm Lip(U_c)} =:C_1.
\end{equation}
As a result, since $\|\zeta'\|_{C^0(\R)}\leq 4$ and $T$ is almost an isometry (see Eq.\,\eqref{DT}), we can bound
\begin{equation}\label{k1}
|K_1(x)| \leq C_2\int_\Omega \frac{|\omega(y)|}{|x-y|^2}\,\dd y\qquad \text{ for } x \in U_b,
\end{equation}
where the constant $C_2=C(\|\F\|_{\rm Lip(U_b)},1/d_4)$. The same bound remains valid with the indices $i,j$ interchanged. For the $K_2$ term, one observes that
\begin{align}
\na_{x,j} \frac{(Tx-Ty)^p}{|Tx-Ty|^3} &= \frac{\na_j (Tx)^p}{|Tx-Ty|^3} - 6\sum_{q}\frac{(\txty)^p(\txty)^q\na_j(Tx)^q}{|Tx-Ty|^5}\nonumber\\
&=\frac{\Xi_p^j(x)}{|\txty|^3}-6\sum_q \frac{(\txty)^p(\txty)^q\,\Xi_q^j(x)}{|Tx-Ty|^5}.
\end{align}

Hence, the symmetric gradient of $J_{211}$ equals to 
\begin{align}\label{symmetric grad of J211}
&\frac{1}{2} \big(\na_j[J_{211}]^i + \na_i[J_{211}]^j\big)(x)\nonumber\\
=\,&\sum_{klpq} \frac{\e^{kli}}{4\pi} \int_\Omega\chi(y)\zeta\bigg(\frac{|\txty|}{d_4}\bigg)\omega^l(y)\Xi^k_p(y)\bigg[\frac{\Xi_p^j(x)}{|\txty|^3}-6 \frac{(\txty)^p(\txty)^q\,\Xi_q^j(x)}{|Tx-Ty|^5}\bigg]\,\dd y\nonumber\\
&\,+ \sum_{klpq} \frac{\e^{klj}}{4\pi} \int_\Omega\chi(y)\zeta\bigg(\frac{|\txty|}{d_4}\bigg)\omega^l(y)\Xi^k_p(y)\bigg[\frac{\Xi_p^i(x)}{|\txty|^3}-6 \frac{(\txty)^p(\txty)^q\,\Xi_q^i(x)}{|Tx-Ty|^5}\bigg]\,\dd y\nonumber\\
&\, + K_3(x),
\end{align}
where $K_3$ has the same bound \eqref{k1} as for $K_1$. The first terms in the second and third lines above have nice cancellation properties, thanks to the following observation:
\begin{lemma}\label{lemma: algebraic}
For some $C_3=C(\|\na^2\F\|_{C^0(U_b)})$, there holds
\begin{equation}
\sum_{ijkp}\e^{kli}\big(\Xi^k_p(y) \Xi^j_p(x)+\Xi^k_p(y) \Xi^i_p(x)\big) \leq C_3|x-y|
\end{equation}
for $x,y$ sufficiently close in $U_b$.
\end{lemma}
\begin{proof}
Using $\OO^{-1}=\OO^{\top}$ and the definition of $\Xi$ in \eqref{Xi}, we have
\begin{align}\label{xixi}
\Xi^k_p(y) \Xi^j_p(x) &= \delta^k_i +\sum_{\gamma,\eta=1}^2\delta^k_3 \delta^i_3\delta^\gamma_\eta \na_\gamma\F(y-x_b)\na_\eta\F(x-x_b) \nonumber\\
&\quad - \Big(\delta^i_3\na_k\F(x-x_b) +\delta^k_3 \na_i\F(y-x_b)\Big)\nonumber\\
&= \delta^k_i +\sum_{\gamma,\eta=1}^2\delta^k_3 \delta^i_3\delta^\gamma_\eta \na_\gamma\F(y-x_b)\na_\eta\F(x-x_b)\nonumber\\
&\quad - \Big(\delta^i_3\na_k\F(x-x_b) +\delta^k_3 \na_i\F(x-x_b)\Big) + \delta^k_3 \big(\na_i\F(x-x_b)-\na_i\F(y-x_b)\big).
\end{align}
The first three terms on the right-hand side are symmetric in $i$ and $k$; hence, multiplying with $\e^{kli}$ and symmetrising over $i$, $j$ yield zero. For the last term, one may use the definition of $T$ and Taylor expansion to deduce 
\begin{equation}
\big|\delta^k_3 \big(\na_i\F(x-x_b)-\na_i\F(y-x_b)\big)\big| \leq C_3 |\txty| = C_4|x-y|\qquad \text{ for } x,y\in U_b.
\end{equation}
Hence the assertion follows.  \end{proof}

The above lemma implies that 
\begin{align}\label{f}
&\bigg|\sum_{klpq} \frac{\e^{kli}}{4\pi} \int_\Omega\chi(y)\zeta\bigg(\frac{|\txty|}{d_4}\bigg)\omega^l(y)\Xi^k_p(y)\frac{\Xi_p^j(x)}{|\txty|^3} \nonumber\\
& \qquad\qquad\qquad\qquad+ \sum_{klpq} \frac{\e^{klj}}{4\pi} \int_\Omega\chi(y)\zeta\bigg(\frac{|\txty|}{d_4}\bigg)\omega^l(y)\Xi^k_p(y)\frac{\Xi_p^i(x)}{|\txty|^3}\bigg|\nonumber\\
&\qquad\qquad\qquad\qquad\leq  C_2\int_\Omega \frac{|\omega(y)|}{|x-y|^2}\,\dd y,
\end{align}
which is the same bound as for $K_1$, $K_3$. For the remaining terms (denoted by $\mathscr{R}$) in Eq.\,\eqref{symmetric grad of J211}, let us introduce the short-hand notation
\begin{eqnarray}
&&\Psi^\sharp(x,y):=\Xi(y)\cdot(\txty),\\
&&\Psi^\flat(x,y):=\Xi(x)\cdot(\txty).
\end{eqnarray}
Thus,
\begin{align}\label{g}
\mathscr{R}&\equiv\sum_{klpq}\frac{\e^{kli}}{4\pi}\int_\Omega\chi(y)\zzz\omega^l(y)\Xi^k_p(y)\bigg[-6\frac{(\txty)^p(\txty)^q\Xi^j_q(x)}{|\txty|^5}\bigg]\,\dd y\nonumber\\
&\quad + \sum_{klpq}\frac{\e^{klj}}{4\pi}\int_\Omega\chi(y)\zzz\omega^l(y)\Xi^k_p(y)\bigg[-6\frac{(\txty)^p(\txty)^q\Xi^i_q(x)}{|\txty|^5}\bigg]\,\dd y\nonumber\\
&=-\frac{3}{2\pi} \sum_{kl} \int_\Omega \chi(y)\zzz \bigg\{\frac{\e^{kli}(\Psi^\sharp)^k(\Psi^\flat)^j + \e^{klj}(\Psi^\sharp)^k(\Psi^\flat)^i}{|\txty|^5} \bigg\}\omega^l(y)\,\dd y\nonumber\\
&= -\frac{3}{2\pi} \int_\Omega \chi(y)\zzz \bigg\{\frac{\sh\times\omega(y) \otimes\fl + \fl\otimes\sh\times\omega(y)}{|\txty|^5} \bigg\}^{ij}\,\dd y.
\end{align}
Here and throughout, the notation for tensor product is understood as follows:
\begin{equation*}
\{a \times b \otimes c\}^{ij} := (a\times b)^i c^j \qquad \text{ for } a,b,c\in\R^3, \, i,j\in\{1,2,3\}.
\end{equation*}
We further notice that  
\begin{equation}\label{geometric observation}
[a\times b \otimes c + b\otimes c \times a]: (d\otimes d) = 2\langle c,d\rangle \det (a, b, d)\qquad \text{ for } a,b,c,d \in \R^3,
\end{equation}
where $\det (a,b,d)$ is the determinant of the $3\times 3$ matrix with columns $a,b$ and $d$ in order. Hence, in view of Eqs.\,\eqref{symmetric grad of J211}\eqref{f}\eqref{g} and \eqref{geometric observation} and Lemma \ref{lemma: algebraic}, one obtains
\begin{align}\label{k5}
&\Big|\int_\Omega\frac{\na J_{211}(x) + \na^\top J_{211}(x)}{2} : \omega(x)\otimes\omega(x)\,\dd x\Big|\nonumber\\
\leq\,& \underbrace{\frac{3}{\pi} \int_\Omega\int_\Omega\chi(y)\zzz \frac{\big|\langle\fl(x,y), \omega(x)\rangle\big|\cdot\big|\det\big(\sh(x,y), \omega(y), \omega(x)\big)\big|}{|\txty|^5}\,\dd y\,\dd x}_{\equiv \, K_5} + K_4,
\end{align}
where, for some $C_5=C(\|\F\|_{C^2(\overline{\Omega})}, 1/d_4)$, there holds
\begin{equation}\label{k4}
|K_4| \leq C_5\int_\Omega|\omega(x)|^2\int_{U_b}  \frac{|\omega(y)|}{|x-y|^2}\,\dd y\,\dd x.
\end{equation}

It remains to bound $K_5$. The key is to explore the geometric meaning of the determinant, as in Constantin--Fefferman \cite{cf} and Constantin \cite{constantin}. This is achieved by the following lemmas. Let us adopt the notation
\begin{equation}
\mathcal{R}_\F (z;w):= \begin{bmatrix}
\na_1\F(z) w^3\\
\na_2\F(z) w^3\\
\na_1\F(z)w^1 + \na_2\F(z)w^2 + |\na\F(z)|^2 w^3
\end{bmatrix}.
\end{equation}
Then we have
\begin{lemma}\label{lemma: K5, 1}
The determinant term in \eqref{k5} satisfies
\begin{align}\label{i}
&\frac{\Big|\det\Big(\sh(x,y),\omega(y),\omega(x)\Big)\Big|}{|\txty|}\nonumber\\
&\qquad\qquad\simeq\bigg\{\Big|\det\,\Big(\xyhat, \omega(y), \omega(x)\Big)\Big| + \Big|\det\,\Big(\frac{\mathcal{R}_\F(y;x-y)}{|x-y|}, \omega(x),\omega(y)\Big)\Big| \bigg\}.
\end{align}
\end{lemma}
Here, recall the notation $\xyhat:=(x-y)/|x-y|$; also, we write $A \simeq B$ to mean that $C^{-1}A \leq B \leq CA$ for a universal constant $C$.

\begin{proof}
	We make a detailed analysis of the term $\sh$. By Taylor expansion and $\OO^{-1}=\OO^\top$ one may deduce
\begin{align}\label{sharp}
[\sh (x,y)]^i&=\sum_{j} \Xi^i_j(y) (\txty)^j\nonumber\\
&=\sum_{jk\eta\gamma} \Big(\OO^i_j + \delta^i_3\na_\gamma\F(y)\OO^\gamma_j\Big)\Big(\OO^k_j + \delta^k_3\na_\eta\F(y)\OO^\eta_j\Big)(x-y)^k+ \smallo(|x-y|)\nonumber\\
&=(x^i-y^i) + \delta^i_3\Big\{\na_1\F(y)(x^1-y^1) + \na_2\F(y)(x^2-y^2) + |\na \F(y)|^2 (x^3-y^3)\Big\}\nonumber\\
&\qquad\qquad + \na_i\F(y) (x^3-y^3)+ \smallo(|x-y|);
\end{align}
Equivalently,
\begin{align}\label{psi sharp}
\sh(x,y)&= (x-y) + \begin{bmatrix}
\na_1\F(y) (x^3-y^3)\\
\na_2\F(y) (x^3-y^3)\\
\na_1\F(y)(x^1-y^1) + \na_2\F(y)(x^2-y^2) + |\na \F(y)|^2 (x^3-y^3)
\end{bmatrix}\nonumber\\
&=:(x-y) + \mathcal{R}_\F(y;x-y).
\end{align}
On the other hand, by shrinking $d_4>0$ if necessary, we conclude from Eq.\,\eqref{DT matrix} that
\begin{equation}\label{TTT}
\frac{1}{2}|x-y| \leq |\txty| \leq 2|x-y|.
\end{equation}
Hence the assertion follows.  \end{proof}

By analogous arguments, we have

\begin{lemma}\label{lemma K5,2}
\begin{align}\label{j}
\frac{|\fl(x,y)|}{|\txty|} \simeq 1+\frac{|\mathcal{R}_\F(x;x-y)|}{|x-y|}.
\end{align}
\end{lemma}

\begin{proof}
A computation similar to \eqref{sharp} gives us
\begin{align}\label{psi flat}
\fl(x,y)&= (x-y) + \begin{bmatrix}
\na_1\F(x) (x^3-y^3)\\
\na_2\F(x) (x^3-y^3)\\
\na_1\F(x)(x^1-y^1) + \na_2\F(x)(x^2-y^2) + |\na \F(x)|^2 (x^3-y^3)
\end{bmatrix}\nonumber\\
&=:(x-y) + \mathcal{R}_\F(x;x-y).
\end{align}
The assertion follows immediately from Eq.\,\eqref{TTT}. \end{proof}

Now, utilising the crucial geometric observation by Constantin \cite{constantin} and Constantin--Fefferman \cite{cf}, we can finalise the estimate for $K_5$. This is the first place where we need the geometric condition in the hypotheses of Theorem \ref{thm: main}.

\begin{lemma}\label{lemma: K5, 3}
Under the assumption of Theorem \ref{thm: main}, {\it i.e.}, the turning angle of vorticity
\begin{equation*}
\theta(x,y) := \angle \Big(\what(x), \what(y)\Big)
\end{equation*}
satisfies
\begin{equation*}
|\sin\theta(x,y)| \leq C_6 \sqrt{|x-y|}
\end{equation*}
for a universal constant $C_6>0$, we can find $C_7=C(C_6, \|\F\|_{C^1(\overline{\Omega})})$ such that
\begin{equation}\label{K5 final bound}
|K_5| \leq C_7 \int_\Omega |\omega(x)|^2 \int_{U_b} \frac{|\omega(y)|}{|x-y|^{5/2}}\,\dd y \,\dd x.
\end{equation}
\end{lemma}

\begin{proof}
In view of Lemmas \ref{lemma: K5, 1} and \ref{lemma K5,2}, substituting Eqs.\,\eqref{i}\eqref{j} into \eqref{k5}, we have:
\begin{align}\label{k5'}
|K_5| &\simeq \int_\Omega |\omega(x)|^2 \int_\Omega \chi(y)\zzz \frac{|\omega(y)|}{|x-y|^3} \Big(1+\frac{|\mathcal{R}_\F(x;x-y)|}{|x-y|}\Big) \times \nonumber\\
&\qquad\times\bigg\{\Big|\det\Big(\xyhat, \what(y),\what(x)\Big)\Big| + \Big|\det\Big(\frac{\mathcal{R}_\F(y;x-y)}{|x-y|}, \what(x),\what(y)\Big)\Big| \bigg\} \,\dd y\,\dd x.
\end{align}

Now we invoke the geometric observation by Constantin \cite{constantin} and Constantin--Fefferman \cite{cf} (also see Beir\~{a}o da Veiga--Berselli \cite{bb2} and the references cited therein): Consider the expression
$$
\det \big(\widehat{a},\what(x),\what(y)\big)
$$
for any unit vector $\widehat{a}\in\R^3$. It is the volume of the parallelepiped spanned by the sides $\widehat{a}$, $\what(x)$ and $\what(y)$, hence equals to 
$$
\det\Big(\widehat{a}, \what(x), {\rm pr}_{[\what(x)]^\perp}\what(y)\Big).
$$
Here ${\rm pr}_{[\what(x)]^\perp}(\cdot)$ denotes the orthogonal projection onto the subspace perpendicular to $\what(x)$. Moreover, as $|\what(y)|=1$, one has 
\begin{align}\label{l}
\Big|\det\Big(\widehat{a}, \what(x), {\rm pr}_{[\what(x)]^\perp}\what(y)\Big)\Big|&\leq\Big|{\rm pr}_{[\what(x)]^\perp}\what(y)\Big|\nonumber\\
&\leq |\sin \theta(x,y)|.
\end{align}
Finally, it is clear that
\begin{equation}
\frac{|\mathcal{R}_\F(\bullet; x-y)|}{|x-y|} \leq \sqrt{3} \|\na \F\|_{C^0(\overline{\Omega})}.
\end{equation}
Therefore, we complete the proof in view of \eqref{k5'} and by considering $\widehat{a}=\xyhat$ in \eqref{l}.  \end{proof}

We conclude this subsection with the following bound for the contribution of $J_{211}$ to the vortex stretching term:

\begin{proposition}\label{proposition: J211 vortex stretching}
Under the assumption of Theorem \ref{thm: main}, 
\begin{align}
&\Big|\int_\Omega\frac{\na J_{211}(x) + \na^\top J_{211}(x)}{2} : \omega(x)\otimes\omega(x)\,\dd x\Big|\nonumber\\
&\qquad \qquad \qquad \leq  C_8\bigg\{\int_\Omega|\omega(x)|^2\int_{U_b}  \frac{|\omega(y)|}{|x-y|^2}\,\dd y\,\dd x +\int_\Omega |\omega(x)|^2 \int_{U_b} \frac{|\omega(y)|}{|x-y|^{5/2}}\,\dd y \,\dd x\bigg\}
\end{align}
where $C_8=C(\|\F\|_{C^2(\overline{\Omega})}, 1/d_4)$.
\end{proposition}
\begin{proof}
Immediate from Lemma \ref{lemma: K5, 3} and Eqs.\,\eqref{k5}, \eqref{k4}.   \end{proof}

\subsection{Estimates for $J_{212}$}

The computation for $J_{212}$ is similar to that for $J_{211}$ in Sect.\,4.5. Recall from Sect.\,4.4:
\begin{align*}
[J_{212}(x)]^i &= \sum_{klp}\frac{\e^{kli}}{2\pi}\int_\Omega\chi(y)\zzz \frac{(\txtys)^p}{|\txtys|^3} \sigma_p \Xi^k_p(y)\,\dd y
\end{align*}
for $x\in U_b$. Then,  
\begin{align*}
\na_j [J_{212}]^i(x)&= \sum_{klpq}\sum_{\gamma\eta} \frac{\e^{kli}}{2\pi} \Bigg\{ \int_\Omega \frac{2}{d_4} \chi(y) \zeta'\bigg(\frac{|Tx-Ty|}{d_4}\bigg) \frac{(Tx-\tyst)^q\sigma_q}{|\txtys|}\,\Xi^j_q(y)\times\nonumber\\
&\quad\times\frac{(Tx-\tyst)^p\sigma_p}{|\txtys|^3}\,\Xi^k_p(y)\omega^l(y) \,\dd y\nonumber\\
&\quad+ \int_\Omega \chi(y)\zeta\bigg(\frac{|Tx-Ty|}{d_4}\bigg)\na_{x,j}\bigg[\frac{(Tx-\tyst)^p\sigma_k}{|\txtys|^3}\,\Xi^k_p(y)\bigg] \omega^l(y)\,\dd y \Bigg\}
\end{align*}
by a direct computation. Using similar arguments as for $J_{211}$ (in particular, Lemma \ref{lemma: algebraic}), we can deduce 
\begin{align}\label{m}
\Big|\frac{1}{2}\int_\Omega\Big(\na_j [J_{212}]^i(x) + \na_i [J_{212}]^j(x)\Big): \omega(x)\otimes \omega(x)\,\dd x\Big| \leq K_6 + K_7, 
\end{align}
where the ``nice'' term is bounded by
\begin{equation}\label{n}
K_6 \leq C_9 \int_\Omega |\omega(x)|^2\int_{U_b} \frac{|\omega(y)|}{|x-y|^2}\,\dd y\, \dd x
\end{equation}
for some constant $C_9$ depends only on $\|\F\|_{C^2(\overline{\Omega})}$. The ``bad'' term in Eq.\,\eqref{m} equals to
\begin{align}
K_7 &=
\frac{C_{10}}{2}\sum_{ijkpql} \int_\Omega\int_\Omega \chi(y)\zzz \times\nonumber\\
 &\qquad\qquad\times\bigg| \frac{\e^{kli}(\txtys)^p(\txtys)^q\omega^l(y)\sigma_k\Xi^j_q(x)\Xi^k_p(y) }{|\txtys|^5}\omega^i(x)\omega^j(x) \bigg|\,\dd y\,\dd x,
\end{align}
where $C_{10}$ is a universal constant. In the above these symbols are introduced:
\begin{eqnarray}
&&\shh\equiv\shh(x,y) :=M{\Xi}(y) \cdot \big(Tx-\tyst\big),\\
&&\fll\equiv\fll(x,y) := {\Xi}(x) \cdot \big(Tx-\tyst\big),
\end{eqnarray}
$z^\star$ denotes the reflection of $z \in \R^3_+$ across the boundary as usual, as well as
\begin{equation}
M=\begin{bmatrix}
1&0&0\\
0&1&0\\
0&0&-1
\end{bmatrix}.
\end{equation}
Thus, using the geometric observation in \cite{cf, constantin}, we find:
\begin{equation}\label{k7}
K_7= C_{10}\sum_{ijkpql} \int_\Omega\int_\Omega \chi(y)\zzz \bigg| \frac{\langle\fll, \omega(x)\rangle \det \big(\shh, \omega(y), \omega(x)\big)}{|\txtys|^5} \bigg|\,\dd y\,\dd x,
\end{equation}
which is analogous to $K_5$ in Eq.\,\eqref{k5} in Sect.\,4.5.

However, it is clear that
\begin{equation}\label{o}
\frac{|\fll|}{|\txtys|} \leq |\Xi(x)| \leq C_{11} = C(\|\F\|_{C^1(U_b)});
\end{equation}
in addition, assuming the hypothesis in Theorem \ref{thm: main}, one obtains
\begin{equation}\label{p}
\bigg|\det\bigg(\frac{\shh}{|\txtys|}, \what(y), \what(x)\bigg)\bigg| \leq C_{12} \sqrt{|x-y|}
\end{equation}
for $C_{12}=C(\|\F\|_{C^1(U_b)})$. Indeed, one easily bounds
$$
\Big|\frac{\shh}{|\txtys|}\Big|
\leq |M\OO| + |M\OO \na\F|,
$$
where both $M,\OO$ are orthogonal matrices, in view of \eqref{Xi}.  Putting together the estimates in Eqs.\,\eqref{m}\eqref{n}\eqref{k7}\eqref{o} and \eqref{p}, we can deduce:

\begin{proposition}\label{proposition: J212 vortex stretching}
Under the assumption of Theorem \ref{thm: main}, we have
\begin{align}
&\Big|\int_\Omega\frac{\na J_{212}(x) + \na^\top J_{212}(x)}{2} : \omega(x)\otimes\omega(x)\,\dd x\Big|\nonumber\\
&\qquad \qquad \qquad \leq  C_{13}\bigg\{\int_\Omega|\omega(x)|^2\int_{U_b}  \frac{|\omega(y)|}{|x-y|^2}\,\dd y\,\dd x +\int_\Omega |\omega(x)|^2 \int_{U_b} \frac{|\omega(y)|}{|x-y|^{5/2}}\,\dd y \,\dd x\bigg\}
\end{align}
where $C_{13}=C(\|\F\|_{C^2(\overline{\Omega})}, 1/d_4)$.
\end{proposition}


\subsection{Estimates for $J_{213}$}\label{subsec J213}
This is a good term, due to the  decay properties of the kernel $\TI$. We recall it from \eqref{j2: three terms}:
\begin{equation*}
[J_{213}(x)]^i = \sum_{kl}\frac{\e^{kli} b^{(i)}_3}{6\pi} \int_\Omega \chi(y)\zzz\frac{\na_k \big[\TI\big(Tx,\tyst\big)\big]\omega^l(y)}{|\txtys|}\,\dd y,
\end{equation*}
where the $\TI$ term is given by \eqref{theta in green matrix}:
\begin{equation*}
\TI\big(Tx,\tyst\big) = \int_0^\infty e^{a^{(i)}|\txtys|s} \, \frac{\xi^3 + b_3^{(i)}s}{\big[1+2\langle\bbi,\xi\rangle s + s^2\big]^{3/2}}\,\dd s,
\end{equation*}
and
\begin{equation*}
\xi = \frac{\txtys}{|\txtys|}.
\end{equation*}

\begin{lemma}\label{lemma: Theta kernel}
For the regular oblique derivative bondary condition \eqref{homog bc}, {\it i.e.}, if for each $i\in\{1,2,3\}$ one has
\begin{equation}
b^{(i)}_3 > 0,\qquad \n = \frac{\p}{\p x^3} \text{ on } \Sigma,
\end{equation}
$\TI\big(Tx,\tyst\big)$ is a smooth function in $x$ and $y$.
\end{lemma}

\begin{proof}
First of all, note that $|\txtys|\neq 0$ from the definition of the boundary-straightening map $T=T_b$; hence $\xi$ is well defined and smooth for all $x$ and $y$, so are
\begin{equation}\label{r}
\na_{x,j} |\txtys| = 2\frac{(\txtys)^k \big[\Xi^j_k(x)+\smallo(|x-x_b|)\big]}{|\txtys|}
\end{equation}
and
\begin{align}\label{s}
\na_{x,j}\xi^k &= \frac{\Xi^j_k(x)+\smallo(|x-x_b|)}{|\txty|} \nonumber\\
&\qquad - \frac{2(\txtys)^k(\txtys)^l\big[\Xi^j_l(x) + \smallo(|x-x_b|)\big]}{|\txtys|^3}.
\end{align}
Next, we can compute
\begin{align}
\na_{x,j} \TI\big(Tx,\tyst\big) &= \int_0^\infty e^{a^{(i)}|\txtys|s} \frac{(\xi^3+b^{(i)}_3s)a^{(i)}s\na_{x,j}|\txtys|}{\big[1+2\langle\bbi,\xi\rangle s + s^2\big]^{3/2}}\,\dd s\nonumber\\
&+\int_0^\infty e^{a^{(i)}|\txtys|s} \frac{\na_{x,j}\xi^3}{\big[1+2\langle\bbi,\xi\rangle s + s^2\big]^{3/2}}\,\dd s\nonumber\\
&-\int_0^\infty e^{a^{(i)}|\txtys|s} \frac{(\xi^3+b_3^{(i)}s)\langle{\bbi,\na_{x,j}\xi}\rangle}{\big[1+2\langle\bbi,\xi\rangle s + s^2\big]^{5/2}}\,\dd s,
\end{align}
and by a simple induction, for any multi-index $\alpha \in \mathbb{N}^\mathbb{N}$ we have
\begin{equation}\label{v}
\na^\alpha_x \TI\big(Tx,\tyst\big) = \int_0^\infty e^{a^{(i)}|\txtys|s} \mathscr{P}_\alpha(x,y,s)\,\dd s,
\end{equation}
where $\mathscr{P}_\alpha(x,y,s)$ is a linear combination of polynomials in $s$. The coefficients of such polynomials are products of components of $\xi$, $\na_x |\txtys|, \na_x\xi$ and $(1+2\langle\bbi,\xi\rangle s + s^2)^k$ for $k \leq -3/2$. In view of \eqref{r}, \eqref{s} and the assumptions $a^{(i)}\leq 0$, $b^{(i)}_3>0$, $|\bbi|=1$ for the regular oblique derivative condition, the integral \eqref{u} converges for any multi-index $\alpha$, and is continuous in the $x$-variable. Finally, the derivatives $\na^\alpha_y \TI\big(Tx,\tyst\big)$ differs from \eqref{u} only by multiplications of the constant matrix $M={\rm diag}\,(1,1,-1)$. Hence the assertion follows.  \end{proof}

As a consequence, the gradient of $J_{213}$:
\begin{align}
\na_j[J_{213}(x)]^i &= \sum_{kl}\frac{\e^{kli}b_3^{(i)}}{6d_4\pi} \int_\Omega \chi(y)\zeta'\bigg(\frac{|\txty|}{d_4}\bigg)\times\nonumber\\
&\qquad\times\frac{\big(\na_{x,j}|\txtys|\big)\na_{y,k} \big[\TI\big(Tx,\tyst\big)\big]\omega^l(y)}{|\txtys|}\,\dd y\nonumber\\
&\,+\sum_{kl}\frac{\e^{kli}b_3^{(i)}}{6\pi}\int_\Omega \chi(y)\zzz \frac{\na_{x,j}\na_{y,k}\big[\TI\big(Tx,\tyst\big)\big]\omega^l(y)}{|\txtys|}\,\dd y\nonumber\\
&\,+ \sum_{kl}\frac{\e^{kli}b_3^{(i)}}{6\pi}\int_\Omega \chi(y)\zzz \frac{\na_{y,k}\big[\TI\big(Tx,\tyst\big)\big]\big(\na_{x,j}|\txty|\big)\omega^l(y)}{|\txtys|}\,\dd y
\end{align}
satisfies good bounds (so does its symmetrisation), because
$$
|\na|\txtys|| \leq C_{14} = C(\|\F\|_{C^1(U_b)})
$$ and $\TI(Tx,\tyst) \in C^\infty$ by Eq.\,\eqref{r} and Lemma \ref{lemma: Theta kernel}. More precisely, 
\begin{proposition}\label{proposition: J213 vortex stretching}
Under the assumption of Theorem \ref{thm: main}, we have
\begin{align}
\Big|\int_\Omega\frac{\na J_{213}(x) + \na^\top J_{213}(x)}{2} : \omega(x)\otimes\omega(x)\,\dd x\Big| \leq  C_{15} \int_\Omega |\omega(x)|^2\int_\Omega \frac{|\omega(y)|}{|x-y|}\,\dd y\,\dd x,
\end{align}
where $C_{15}=C(\|\F\|_{C^1(\overline{\Omega})}, 1/d_4, \bbi,  a^{(i)})$.
\end{proposition}

The above proposition can be proved without using the hypothesis on vorticity directions. 

\subsection{Estimates for $J_{22}$}
Next, $J_{22}$ is also a good term (recall from Eq.\,\eqref{j2: three terms}):
\begin{equation}
[J_{22}(x)]^i=-\sum_{kjl}\e^{klj}\int_\Omega \na_k\bigg[\chi(y)\zeta\bigg(\frac{|Tx-Ty|}{d_4} \bigg) \bigg] \GG_{ij}(Tx,Ty) \omega^l(y)\,\dd y.
\end{equation}
Clearly, by the definition of $\chi$ and $\zeta$, 
\begin{equation}
\bigg| \na_k\bigg[\chi(y)\zeta\bigg(\frac{|Tx-Ty|}{d_4} \bigg) \bigg]\bigg| \leq C_{16} = C(\|\F\|_{C^1(\overline{\Omega})}, 1/d_4).
\end{equation}
In addition, in light of  Lemma \ref{lemma: Theta kernel},
\begin{equation*}
\GG_{ij}(Tx,Ty)=\frac{\delta_{ij}}{4\pi}\bigg\{\frac{1}{|\txty|} - \frac{1}{|\txtys|}\bigg(1+\frac{2b_3^{(i)}}{3}\TI\big(Tx,\tyst\big)\bigg) \bigg\}
\end{equation*}
has a singularity of order $-1$, {\it i.e.},
\begin{equation}
|\GG_{ij}(Tx, Ty)| \leq C_{17} \frac{1}{|x-y|}
\end{equation} 
for some constant $C_{17}=C(\|\F\|_{C^1(\overline{\Omega})}, \bbi,  a^{(i)})$. Therefore, we may easily deduce
\begin{proposition}\label{proposition: J22 vortex stretching}
Under the assumption of Theorem \ref{thm: main}, we have
\begin{align}
\Big|\int_\Omega\frac{\na J_{22}(x) + \na^\top J_{22}(x)}{2} : \omega(x)\otimes\omega(x)\,\dd x\Big| \leq  C_{18} \int_\Omega |\omega(x)|^2\int_\Omega \frac{|\omega(y)|}{|x-y|^2}\,\dd y\,\dd x,
\end{align}
where $C_{18}=C(\|\F\|_{C^1(\overline{\Omega})}, 1/d_4, \bbi,  a^{(i)})$.
\end{proposition}

Again, in Proposition \ref{proposition: J22 vortex stretching} we do not need the hypothesis on vorticity direction alignment.

\subsection{Estimates for $J_{23}$: the Boundary Term} One of the main new features of this work is the analysis of the boundary term, reproduced below from Eq.\,\eqref{j2: three terms}:
\begin{equation*}
[J_{23}(x)]^i = -\sum_j \int_\Sigma \chi(y)\zzz \GG_{ij}(Tx, Ty) (\omega \times \n)^j\,\dd\htwo(y).
\end{equation*}
In the literature the geometric regularity conditions for the weak solutions to the Navier--Stokes equations are usually studied on the whole space $\R^3$, {\it i.e.}, in the absence of physical boundaries of the fluid domain. In Beir\~{a}o da Veiga--Berselli \cite{bb2} and Beir\~{a}o da Veiga \cite{bb3} the boundary conditions were first considered. Therein the slip-type condition
\begin{equation}\label{non physical bc}
\omega \times \n = 0 \qquad \text{ on } [0,T^\star[\times\Sigma
\end{equation}
was imposed (which were first studied by Solonnikov--\u{S}\u{c}adilov \cite{ss}), so that the boundary term vanishes: $J_{23} \equiv 0$. It is a very strong condition on the geometry of the vortex structure $\Sigma$, which entails the vorticity to be perpendicular to the boundary of the fluid domain. 

In our current work the condition \eqref{non physical bc} is not required. Instead, we only require that the sine of the turning angle of vorticity $\theta$ remains $(1/2)$-H\"{o}lder {\em up to the boundary}, {\it i.e.}, the hypotheses of Theorem \ref{thm: main}. We shall establish:
\begin{proposition}\label{proposition: J23 vortex stretching}
Under the assumption of Theorem \ref{thm: main}, we have
\begin{align}
&\Big|\int_\Omega\frac{\na J_{23}(x) + \na^\top J_{23}(x)}{2} : \omega(x)\otimes\omega(x)\,\dd x\Big| \nonumber\\
 &\qquad\qquad \leq  C_{19} \bigg\{\int_\ooo |\omega(x)|^2\int_\Sigma \frac{|\omega(y)|}{|x-y|^{3/2}}\,\dd y\,\dd x  \nonumber\\
&\qquad\qquad\qquad + \int_\ooo |\omega(x)|^2\int_\Sigma \frac{|\omega(y)|}{|x-y|^{1/2}}\,\dd y\,\dd x  \nonumber\\
 &\qquad\qquad\qquad+ \int_\ooo |\omega(x)|^2\int_\Sigma {|\omega(y)|}{|x-y|^{1/2}}\,\dd y\,\dd x  \bigg\},
\end{align}
where $C_{19}=C(\|\F\|_{C^1(\overline{\ooo})}, 1/d_4, \bbi,  a^{(i)})$.
\end{proposition}

Here and throughout, for $E \subset \R^3$, $\delta >0$, we write
\begin{equation*}
O_\delta (E) := \big\{x+y: |x|<\delta, y \in E\big\}.
\end{equation*}
Also, we recall that $d_3$ defined in Eq.\,\eqref{d_3} satisfies $d_3 \geq 16 d_4$. 
\begin{proof}
	By a direct computation we can get
	\begin{align}\label{k8 k9 k10}
	\na_j [J_{23}(x)]^i &= \frac{1}{2\pi}\sum_{k} \int_\Sigma\chi(y)\zeta'\bigg(\frac{|\txty|}{d_4}\bigg) \frac{(\txty)^k}{|\txty|}\Big(\Xi(z)^j_k\Big)(\omega\times\n)^i(y)\,\dd\htwo(y)\nonumber\\
	&\qquad- \frac{1}{4\pi}\sum_k\int_\Sigma\chi(y)\zzz \frac{(\txty)^k}{|\txtys|^3}\Big(\Xi(z)^j_k \Big)(\omega\times\n)^i(y)\,\dd\htwo(y)\nonumber\\
	&\qquad -\frac{1}{4\pi}\sum_k\int_\Sigma\chi(y)\zzz \frac{\sigma_k(\txtys)^k}{|\txtys|^3}\Big(\Xi(z')^j_k \Big)(\omega\times\n)^i(y)\,\dd\htwo(y)\nonumber\\ 
	&\qquad -\frac{1}{4\pi} \int_\Sigma \chi(y) \zzz \frac{\na_{x,j}\Big[\TI(Tx,\tyst)\Big]}{|\txtys|} (\omega \times \n)^i(y)\,\dd \htwo(y)\nonumber\\
	&=: [K_8(x)]^i_j + [K_9(x)]^i_j+[K_{10}(x)]^i_j+[K_{11}(x)]^i_j,
	\end{align}
where $z$ ($z'$) is a point on the segment connecting $Tx$ and $Ty$ ($\tyst$, resp.), found by the Taylor expansion. We need to bound $\int_\Omega|(K_l + K_l^\top):\omega \otimes \omega|\,\dd x$ for $l\in\{8,9,10,11\}$. 

For this purpose, parallel to the treatments in Sects.\,4.5--4.6, let us define two vector fields:
\begin{equation}
\overline{\Psi}^\sharp (x,y):= \sum_{jk} (\txty)^k\Xi(z)^j_k\frac{\p}{\p x^j},
\end{equation}
and 
\begin{equation}
\overline{\Psi}^\flat (x,y) := \sum_{jk} \sigma_k (\txtys)^k \Xi(z')^j_k\frac{\p}{\p x^j}.
\end{equation}
So, we can compute
\begin{align}
&\int_\Omega\big|(K_8 + K_8^\top):\omega \otimes \omega\big|(x)\,\dd x \nonumber\\
&\,= \frac{1}{\pi} \int_\Omega \int_\Sigma \chi(y)\bigg|\zeta'\bigg(\frac{|\txty|}{d_4}\bigg)\bigg|  \frac{1}{|\txty|}\big|\big\langle\shhh,\omega(x)\big\rangle \big\langle\omega(y)\times\n (y),\omega(x)\big\rangle\big| \,\dd\htwo(y)\,\dd x.
\end{align}
It is crucial to recognise the {\em determinant structure in disguise}:
\begin{align}
\big\langle\omega(y)\times\n (y),\omega(x)\big\rangle &= \sum_{ijk}\e^{ijk}\omega^i(y)\n^j(y)\omega^k(x)\nonumber\\
&=\det\big(\omega(y),\n(y),\omega(x)\big).
\end{align}
Again by the geometric observation due to Constantin \cite{constantin} and Constantin--Fefferman \cite{cf}, one may deduce
\begin{equation}\label{wnw}
\big|\big\langle\omega(y)\times\n (y),\omega(x)\big\rangle \big| \leq |\omega(y)||\omega(x)| |\sin\theta(x,y)|
\end{equation} 
where $\theta(x,y)=\angle(\omega(x),\omega(y))$. In addition, in view of the definition of $\Xi$ (see \eqref{Xi}), clearly
\begin{equation*}
\frac{|\shhh|}{|\txty|} \leq C_{20} = C(\|\F\|_{C^1(U_b)}).
\end{equation*}
Thus, for $C_{21}$ with the same dependence as $C_{20}$, we have
\begin{align}\label{k8}
\int_\Omega\big|(K_8 + K_8^\top):\omega \otimes \omega\big|(x)\,\dd x\leq C_{21} \int_{O_{d_3}(U_b)}|\omega(x)|^2 \int_{\Sigma\cap U_b} |\omega(y)| \sqrt{|x-y|} \,\dd \htwo(y)\,\dd x,
\end{align}
provided that 
\begin{equation}\label{theta assumption}
|\sin\theta(x,y)| \leq \rho^{-1} \sqrt{|x-y|}
\end{equation}
as assumed by Theorem \ref{thm: main}.

The terms $K_9$, $K_{10}$ and $K_{11}$ are estimated in similar manners. Indeed,
\begin{align}\label{k9}
&\int_\Omega\big|(K_9 + K_9^\top):\omega \otimes \omega\big|(x)\,\dd x\nonumber\\
&\qquad \leq C_{22} \bigg|\int_{\Omega}\int_\Sigma\chi(y)\zzz \frac{\big\langle {\shhh}, \omega(x)\big\rangle}{|\txtys|^3}\big\langle\omega(x), \omega(y)\times\n(y)\big\rangle\,\dd \htwo(y)\,\dd x\bigg|\nonumber\\
&\qquad \leq C_{22} \int_{O_{d_3}(U_b)}|\omega(x)|^2 \int_{\Sigma\cap U_b} \frac{|\omega(y)|}{|x-y|^{3/2}} \,\dd \htwo(y)\,\dd x,
\end{align}
thanks to Eqs.\,\eqref{wnw}\eqref{theta assumption}\eqref{TTT} and the simple fact $|\txty|\leq|\txtys|$ for $x,y \in U_b$. Here $C_{22}$ depends only on $\|\F\|_{C^1(U_b)}$ and the hypothesis of Theorem \ref{thm: main}. The bound for $K_{10}$ is a variant of \eqref{k9}: using arguments parallel to those in Sect.\,4.6 (see Proposition \ref{proposition: J212 vortex stretching}), we get
\begin{align}\label{k10}
&\int_\Omega\big|(K_{10} + K_{10}^\top):\omega \otimes \omega\big|(x)\,\dd x\nonumber\\
&\qquad\qquad\qquad \leq C_{23} \int_{O_{d_3}(U_b)}|\omega(x)|^2 \int_{\Sigma\cap U_b} \frac{|\omega(y)|}{|x-y|^{3/2}} \,\dd \htwo(y)\,\dd x,
\end{align}
where $C_{23}$ has the same dependent variables as $C_{22}$. Lastly, for $K_{11}$, let us recall from Lemma \ref{lemma: Theta kernel} that $\TI$ is smooth in its variables; thus, the singularity in this term has order $(-1)$. We can thus conclude
\begin{align}\label{k11}
&\int_\Omega\big|(K_{11} + K_{11}^\top):\omega \otimes \omega\big|(x)\,\dd x\nonumber\\
&\qquad\qquad\qquad \leq C_{24} \int_{O_{d_3}(U_b)}|\omega(x)|^2 \int_{\Sigma\cap U_b} \frac{|\omega(y)|}{|x-y|^{1/2}} \,\dd \htwo(y)\,\dd x,
\end{align}
for $C_{24}$ depending on $\|\F\|_{C^1}, 1/d_4$, $\bbi$, $a^{(i)}$, and $\rho$ as in the hypothesis of Theorem \ref{thm: main}. Therefore, the proof is complete once we collect the estimates in Eqs.\,\eqref{k8}\eqref{k9}\eqref{k10}\eqref{k11} and \eqref{k8 k9 k10}.  \end{proof}

\subsection{Estimates for $J_1$, $J_3$}
The estimate for $J_1$ is not new. As $J_1$ (reproduced below) only involves the interior charts
\begin{equation}
J_1(x) = \sum_{j=1}^3\sum_{a\in\I}\int_{\Omega}  \chi_a(y) \bigg\{\frac{\delta_{ij}}{4\pi|x-y|} (\na\times\omega)^j(y)\bigg\} \zeta\bigg(\frac{|x-y|}{d_3}\bigg)\,\dd y,
\end{equation}
its contribution to $[Stretch]$ can be computed as in the pioneering works by Constantin--Fefferman \cite{cf} and Beir\~{a}o da Veiga--Berselli \cite{bb1}:
\begin{proposition}\label{propn: J1}
Under the assumption of Theorem \ref{thm: main}, there is a constant $C_{25}=C(\|\F\|_{C^2(\overline{\Omega})})$ such that
\begin{align}
&\Big|\int_\Omega\frac{\na J_{1}(x) + \na^\top J_{1}(x)}{2} : \omega(x)\otimes\omega(x)\,\dd x\Big|\leq  C_{25}\int_\Omega |\omega(x)|^2 \int_{\Omega} \frac{|\omega(y)|}{|x-y|^{5/2}}\,\dd y \,\dd x.
\end{align} 
\end{proposition}

For $J_3$, Solonnikov \cite{s2} (also see p.610 and Appendix B, p.626 in Beir\~{a}o da Veiga--Berselli \cite{bb2}, and Lemma \ref{lemma: solonnikov} in this paper) showed that, for sufficiently regular boundary $\Sigma$, the good part of the kernel $\GG^{\rm good}$ satisfies
\begin{equation}\label{delta}
\Big|\na^\alpha_x\na^\beta_y \GG^\good(x,y)\Big| \leq \frac{C_\good}{|x-y|^{|\alpha|+|\beta|+1-\delta}} \qquad \text{ for all } x \neq y \text{ in } \Omega \text{ with } \delta >1/2.
\end{equation}
In fact, the range of $\delta$ depends only on the regularity of the solution to the elliptic system \eqref{Poisson eq}\eqref{homog bc}; as a consequence of the standard Schauder theory, this in turn depends only on the regularity of $\Omega$. Thanks to Eq.\,\eqref{delta}, a direct computation give us:
\begin{proposition}\label{propn: J3}
Under the assumption of Theorem \ref{thm: main}, there is a constant $C_{26}$ such that
\begin{align}
&\Big|\int_\Omega\frac{\na J_{3}(x) + \na^\top J_{3}(x)}{2} : \omega(x)\otimes\omega(x)\,\dd x\Big|\leq  C_{26}\int_\Omega |\omega(x)|^2 \int_{\Omega} \frac{|\omega(y)|}{|x-y|^{5/2}}\,\dd y \,\dd x.
\end{align} 
Here $C_{26}$ depends only on the regularity of $\Omega$.
\end{proposition}

The estimation for $J_3$ is the only place where we possibly need higher regularity of the domain $\Omega$ than $C^2$. In the case of the slip-type boundary conditions \eqref{special bc}, it is shown in \cite{bb2} that $\Omega \in C^{3,\alpha}$ is enough. In our case of the general diagonal oblique derivative conditions \eqref{homog bc} $\Omega \in C^{3,\alpha}$ will also suffice, in view of the Schauder theory for the oblique derivative problem; {\it cf.} Gilbarg--Trudinger, Chapter 6 \cite{gt}. This is true when the coefficients of the boundary conditions ($a^{(i)}$, $\bbi$) are constant.

\subsection{Proof of Theorem \ref{thm: new}}

Finally we are at the stage of proving Theorem \ref{thm: new}.  Let us  first recall the Hardy--Littlewood--Sobolev interpolation inequality ({\it e.g.}, see p.106 in Lieb--Loss \cite{ll}):
\begin{lemma}[Hardy--Littlewood--Sobolev]\label{lemma: hls}
Let $1<p,r<\infty$ and $0<\lambda<n$ satisfy $1/p+\lambda/n+1/r=2$. Let $f\in L^p(\R^n)$ and $h \in L^r(\R^n)$. Then there exists $K=C(n,\lambda,p)$ such that
\begin{equation*}
\bigg|\int_{\R^n}\int_{\R^n} \frac{f(x)h(y)}{|x-y|^\lambda}\,\dd x\,\dd y\bigg| \leq K \|f\|_{L^p(\R^n)}\|h\|_{L^r(\R^n)}. 
\end{equation*}
\end{lemma}

\begin{proof}[Proof of Theorem \ref{thm: new}]
First of all, in view of the localisation procedure in Sect.\,4.1, it suffices to prove the result on each local chart. Thus, without loss of generalities, let us assume $\Omega$ to be bounded in $\R^3$. The unbounded case follows from a partition-of-unity argument.

The proof follows from a standard continuation argument. Suppose that there were some $T \in ]0,T^\star]$ such that the weak solution $u$ is strong on $[0,T[$, but cannot be continued as a strong solution past the time $T$. We shall establish
\begin{equation}
\limsup_{t \uparrow T} \int_\Omega |\omega(t)|^2\,\dd x <\infty
\end{equation}
for any such given $T$. It shows that $u$ can be extended to a strong solution to $[0, T + \delta]$ for some $\delta >0$. This contradicts the maximality of $T$. Hence, $u$ is strong solution on $[0,T^\star[$. 

To this end, by collecting the estimates in Subsections 4.2--4.10 (in particular, Propositions \ref{proposition: J211 vortex stretching}, \ref{proposition: J212 vortex stretching}, \ref{proposition: J213 vortex stretching}, \ref{proposition: J22 vortex stretching}, \ref{proposition: J23 vortex stretching} and \ref{propn: J1})  and recalling Eq.\,\eqref{u} in Lemma \ref{lemma: representation formula}, let us first bound
\begin{align}
[{\rm Stretch}] &= 2 \,\Big|\int_\Omega \na u: \omega \otimes \omega \,\dd x\Big|\nonumber\\
&\leq C_{27} \int_\Omega |\omega(x)|^2 \int_{\Omega} \frac{|\omega(y)|}{|x-y|^{5/2}}\,\dd y \,\dd x +C_{27} \int_\Omega |\omega(x)|^2\int_{\Sigma} \frac{|\omega(y)|}{|x-y|^{3/2}}\,\dd\htwo(y)\dd x \nonumber\\
&=: I_{\Omega} + I_{\Sigma}
\end{align}
where $C_{27}=C(\Omega, \bbi, a^{(i)})$; note that $d_4$ depends only on the geometry of $\Omega$ and the partition-of-unity, so we do not write it explicitly here.

We first control the bulk term $I_\Omega$. By Lemma \ref{lemma: hls} above, we get 
\begin{equation*}
I_{\Omega} \leq C_{28} \bigg(\int_\Omega |\omega|^3\,\dd x\bigg)^{\frac{2}{3}}\bigg(\int_\Omega |\omega|^2\,\dd x\bigg)^{\frac{1}{2}},
\end{equation*}
where $C_{28}$ equals the product of $C_{27}$ and the constant in the Hardy--Littlewood--Sobolev inequality. In addition, thanks to the interpolation inequality, there holds
\begin{equation*}
\bigg(\int_\Omega |\omega|^3\,\dd x\bigg)^{\frac{2}{3}} \leq \bigg(\int_\Omega |\omega|^2\,\dd x\bigg)\bigg(\int_\Omega |\omega|^6\,\dd x\bigg)^{1/6}
\end{equation*}
and, by the Sobolev inequality,
\begin{equation*}
\|\omega\|_{L^6(\Omega)} \leq C_{29}\Big(\|\omega\|_{W^{1,2}(\Omega)}\Big)
\end{equation*}
where $C_{29}$ depends only on the geometry of $\Omega$. Thus, by Young's inequality we  conclude:
\begin{equation}\label{I Omega}
I_{\Omega} \leq \frac{\e}{2} \int_\Omega |\na \omega|^2\,\dd x + C_{30} \bigg(\int_{\Omega}|\omega|^2\,\dd x\bigg)^2+ C_{30} \bigg(\int_{\Omega}|\omega|^2\,\dd x\bigg),
\end{equation}
with any $\e>0$ and $C_{30}=C(\e, \Omega, \bbi, a^{(i)})$. 

To control the boundary term $I_\Sigma$, by partition of unity and boundary straightening, it suffices to prove for $\Omega=\Sigma \times [0,1]$, $\Sigma = [0,1]^2$. The estimates differ at most by a constant depending only on the geometry of $\Omega$. In addition, denote by $\Sigma_\sigma := [0,1]^2 \times \{\sigma\}$ for $0 \leq \sigma \leq 1$. By Fubini's theorem, we have
\begin{equation}
I_\Sigma = \int_0^1 \int_{\Sigma_\sigma} |\omega(z)|^2 \int_{\Sigma} \frac{|\omega(y)|}{|z-y|^{3/2}}\,\dd\htwo(y)\,\dd\htwo(z)\,\dd \sigma.
\end{equation}
The Hardy--Littlewood--Sobolev inequality (Lemma \ref{lemma: hls}) leads to
\begin{equation}\label{Sigma term}
I_\Sigma \leq K \int_0^1 \bigg(\int_{\Sigma_\sigma}|\omega|^{\frac{8}{3}}\,\dd\htwo\bigg)^{\frac{3}{4}}\bigg(\int_\Sigma
 |\omega|^2\,\dd\htwo\bigg)^{\frac{1}{2}}\,\dd\sigma.
\end{equation}
On the other hand, we have the interpolation inequality
\begin{equation}
\|\omega\|_{L^{8/3}(\Sigma_\sigma)} \leq \|\omega\|_{L^2(\Sigma_\sigma)}^{\frac{1}{2}} \|\omega\|_{L^4(\Sigma_\sigma)}^{\frac{1}{2}},
\end{equation}
the continuous trace map 
$$
W^{1,2}(\Omega) \map W^{1/2,2}(\Sigma_\sigma),
$$
and the continuous Sobolev embedding $$W^{1/2,2}(\Sigma_\sigma) \emb L^4(\Sigma_\sigma).$$ Therefore, utilising the trace and Young's inequalities and taking the essential supremum over $\sigma \in [0,1]$ in Eq.\,\eqref{Sigma term}, we get\begin{equation}\label{I Sigma}
I_{\Omega} \leq \frac{\e}{2} \int_\Omega |\na \omega|^2\,\dd x + C_{31} \bigg(\int_{\Omega}|\omega|^2\,\dd x\bigg)^2+ C_{31} \bigg(\int_{\Omega}|\omega|^2\,\dd x\bigg),
\end{equation}
with  $C_{31}=C(\e, \Omega, \bbi, a^{(i)})$. 

Putting together the estimates \eqref{I Omega}\eqref{I Sigma}, one obtains
\begin{equation}\label{str}
[\text{Stretch}]  
\leq \e \int_\Omega |\na \omega|^2\,\dd x + C_{32} \bigg(\int_{\Omega}|\omega|^2\,\dd x\bigg)^2+ C_{32} \bigg(\int_{\Omega}|\omega|^2\,\dd x\bigg),
\end{equation}
where the constant $C_{32}=C_{30}+C_{31}$.

Now, in view of the differential inequality for the enstrophy \eqref{enstrophy xx}, by choosing $\e=\nu/16$ in Eq.\,\eqref{str} we may deduce
\begin{equation}\label{gronwall}
\frac{\dd}{\dd t}\bigg(\int_\Omega |\omega|^2\,\dd x\bigg) + \frac{\nu}{8} \int_\Omega |\na \omega|^2\,\dd x \leq C_{33}\bigg(\int_\Omega |\omega|^2\,\dd x\bigg) \bigg(1+\int_\Omega |\omega|^2\,\dd x\bigg) + M.
\end{equation}
Here the constant $C_{33}$ depends on $\Omega, \nu, \bbi, a^{(i)}, \beta$, the initial energy $\|u_0\|_{L^2(\Omega)}$ and $M$. Thus, by Gr\"{o}nwall's lemma, 
\begin{align}
\int_\Omega |\omega(T)|^2\,\dd x &\leq \bigg(\int_\Omega |\omega(0)|^2\,\dd x\bigg)\exp\bigg\{C_{33}\int_0^T\int_\Omega |\omega(t,x)|^2\,\dd x\dd t\bigg\}\nonumber\\
&\qquad + \int_0^T \exp\bigg\{C_{33}M\int_s^T\int_\Omega |\omega(t,x)|^2\,\dd x\dd t\bigg\}\,\dd s.
\end{align}
But, by Lemma \ref{lemma: div curl}, the control on $\int_0^T\int_\Omega |\omega|^2\,\dd x\dd t$ is equivalent to that on $\int_0^T\int_\Omega |\na u|^2\,\dd x\dd t$, which is bounded by the energy inequality \eqref{energy xxx}. Hence $\limsup_{t \uparrow T} \int_\Omega |\omega(t)|^2\,\dd x <\infty$. In view of Lemma \ref{lemma: div curl}, it implies $\na u \in L^\infty(0, T; L^2(\Omega; \mathfrak{gl}(3,\R)))$. Substituting this back into Eq.\,\eqref{gronwall} and invoking Lemma \ref{lemma: na na u}, we get $\na u \in  L^2(0, T; H^1(\Omega; \mathfrak{gl}(3,\R)))$ too. Therefore, $u$ can be continued as a strong solution past the time $T$. This contradicts the blowup at $T$. 

The proof is now  complete.    \end{proof}

At the end of this section, we mention the following result {\it \`{a} la} Constantin--Fefferman \cite{cf}, which can be proved by a slight modification of the arguments in Sect.\,4:
\begin{corollary}\label{cor: 1}
Let $\Omega \subset \R^3$ be a sufficiently regular domain. Let $u$ be a weak solution to the Navier-Stokes equations \eqref{ns}\eqref{incompressible} on $[0,T_\star[ \times \Omega$ with the oblique derivative boundary condition \eqref{homog bc}. Assume that the energy estimate in Theorem \ref{thm: energy estimate} is valid for the strong solutions.  Then, if there are  constants $\rho, \Lambda >0$ such that the vorticity turning angle $\theta$ satisfies the following condition:
\begin{equation}
|\sin \theta(t;x,y)|\1_{\{|\omega(t,x)|\geq\Lambda,\, |\omega(t,y)|\geq\Lambda\}} \leq \rho \sqrt{|x-y|} \qquad \text{ for all } t \in [0,T^\star[,\, x,y\in\overline{\Omega},
\end{equation}
then $u$ is also a strong solution on $[0,T^\star[\times\Omega$. 
\end{corollary}


\section{Geometric Regularity Theorem: Proof of Theorem \ref{thm: main}}
In Sect.\,4 we proved the estimates for the system \eqref{Poisson eq} under the homogeneous diagonal oblique derivative boundary condition \eqref{homog bc} with constant coefficients. Now, let us apply the aforementioned result to the regularity problem for the incompressible Navier--Stokes equations under Navier and kinematic boundary conditions. Our crucial observation is that the Navier and kinematic boundary conditions, in suitable local coordinate frames, can be cast into the form of Eq.\,\eqref{homog bc}. Then Theorem \ref{thm: main} follows from Theorem \ref{thm: new}. 

\begin{proof}[Proof of Theorem \ref{thm: main}]
Let us establish the following {\em claim}: Given each boundary point $p \in \Sigma$, we can find a local coordinate chart $U \subset \R^3$ containing $p$ and an orthonormal frame $\{\p/\p x^1, \p/\p x^2, \p/\p x^3\}$ on $U$ with $\{\p/\p x^1, \p/\p x^2\}$ spanning $\G(TU)$ and $\p/\p x^3 =\n$, in which the boundary conditions \eqref{navier bc}\eqref{kinematic bc} takes the form of Eq.\,\eqref{homog bc} (reproduced below):
\begin{equation*}
a^{(i)} u^i + \sum_{j=1}b^{(i)}_j\na_j u^i = 0 \qquad \text{ on } [0,T^\star[\times\Sigma
\end{equation*}
for each $i\in\{1,2,3\}$. 

To see this, we take $\{\p/\p x^1, \p/\p x^2\} \subset \G(T\Sigma)$ to be the {\em principal direction fields}: that is, we require that the second fundamental form $$
\two=-\na \n: \G(TM) \times \G(TM) \map \G(TM^\perp)
$$ 
to be diagonalised with respect to this basis. Such coordinate frames always exist, as $\two$ is a self-adjoint operator on each $T_p\Sigma$. Then, the Navier boundary condition \eqref{navier bc} can be rewritten as follows:
\begin{align}
0 &= \beta u^i + \nu (\na_k u^i+\na_i u^k)\n^k\nonumber\\
&= \beta u^i + \nu \n\cdot \na u^i + \nu \na_i (u \cdot \n) -\sum_{k=1}^3 \nu u^k\na_i\n^k \qquad \text{ for } i\in\{1,2\}.
\end{align}
In regards to the kinematic boundary condition \eqref{kinematic bc}, the third term on the second line above vanishes. Moreover, the fourth term equals
\begin{equation*}
\sum_{k=1}^3\nu \two_{ik}u^k = \nu \kappa_i u^i,
\end{equation*}
where $\kappa_i$ is the $i$-th principal curvature, namely the eigenvalue of $\two$ that corresponds to the eigenvector $\p/\p x^i$. Thus, taking $\n = \p/\p x^3 \in \G(TM^\perp)$, we may conclude that \eqref{navier bc}\eqref{kinematic bc} are equivalent to the following system of boundary conditions:
\begin{eqnarray}
&& (\beta + \nu \kappa_1) u^1 + \nu \na_3 u^1 = 0,\\
&& (\beta + \nu \kappa_2) u^2 + \nu \na_3 u^2 = 0,\\
&& u^3 =0\qquad \text{ on } [0,T^\star[ \times \Sigma.
\end{eqnarray}

Now, let us set (up to normalisations)
\begin{equation*}
a^{(i)}=-\beta-\nu\kappa_i,\quad b^{(i)}_1=b^{(i)}_2=0, \quad b^{(i)}_3=-\nu\qquad \text{ for } i \in \{1,2\}
\end{equation*}
and
\begin{equation*}
a^{(3)}=-1, \qquad b^{(3)}_j = 0 \qquad \text{ for any } j\in\{1,2,3\}
\end{equation*}
to recover Eq.\,\eqref{homog bc}, namely the oblique derivative boundary condition. Note that if $\beta+\nu\kappa_i\neq 0$ for $i\in\{1,2\}$, it is reduced to the Neumann boundary condition
\begin{equation*}
\na_3 u^i = 0 \qquad \text{ on } [0,T^\star[ \times \Sigma.
\end{equation*}

For $\Sigma=$ round spheres, 2-planes and right circular cylinder surfaces, both the mean curvature and the Gauss curvature of the surface are constant, hence $\kappa_1$ and $\kappa_2$ are constant on $\Sigma$. In fact, by elementary differential geometry of surfaces, these are the only embedded/immersed surfaces in $\R^3$ with constant principal curvatures; see Montiel--Ros \cite{mr2}. Therefore, in these cases the Navier and kinematic boundary conditions \eqref{navier bc}\eqref{kinematic bc} can be recast to the homogeneous diagonal oblique boundary derivative conditions with constant coefficients, {\it i.e.}, Eq.\,\eqref{homog bc}. Hence, thanks to Theorem \ref{thm: new}, the proof is now complete.   \end{proof}


Using the proof above, we can deduce the following result from Corollary \ref{cor: 1}:
\begin{corollary}\label{cor: main}
Let $\Omega \subset \R^3$ be one of the following domains: a round ball, a half-space, or a right circular cylindrical duct. Let $u$ be a weak solution to the Navier--Stokes equations \eqref{ns}\eqref{incompressible}\eqref{initial data} with the Navier and kinematic boundary conditions \eqref{navier bc}\eqref{kinematic bc}. Suppose that the vorticity $\omega = \na \times u$ is {\em coherently aligned up to the boundary} in the following sense: there exist constants $\rho,\Lambda>0$ such that
\begin{equation}
|\sin \theta(t; x,y)|\1_{\{|\omega(t,x)|\geq\Lambda,\,|\omega(t,y)|\geq\Lambda\}} \leq \rho \sqrt{|x-y|} \qquad \text{ for all $x,y\in\overline{\Omega}, \,  t<T^\star$}.
\end{equation}
Here the turning angle of vorticity $\theta$ is defined as
\begin{equation*}
\theta(t;x,y):=\angle\big(\omega(t,x),\omega(t,y)\big).
\end{equation*}
Then $u$ is a strong solution on $[0,T^\star[$.
\end{corollary}

It is an interesting problem to study the geometric regularity criteria for weak solutions to the Navier--Stokes equations in general regular domains in $\R^3$ under the Navier and kinematic conditions \eqref{navier bc}\eqref{kinematic bc}. In full generality, we may have difficulty finding the ``nice'' local frames in which Eqs.\,\eqref{navier bc}\eqref{kinematic bc} can be transformed to {\em constant-coefficient} diagonal homogeneous oblique derivative boundary conditions. Thus, to analyse the boundary conditions \eqref{navier bc}\eqref{kinematic bc} on general embedded surfaces in $\R^3$ calls for new ideas. We leave this question for future investigation.

\bigskip
\noindent
{\bf Declaration}. The author declares no conflict of interests.

\noindent
{\bf Acknowledgement}.
Siran Li would like to thank Dr.\,David Poyato for his helpful comments on potential estimates and suggestions for references. Moreover, the author is very grateful to Dr.\,Theodore Drivas for many insightful discussions, and for pointing out mistakes in an earlier version of the paper.

\end{document}